\documentclass[a4paper,10pt,fleqn]{amsart}
\pdfoutput=1

\usepackage[utf8]{inputenc}
\usepackage[T1]{fontenc}

\usepackage{mathabx}
\def\VERT{\vvvert}

\usepackage{todonotes}

\usepackage{microtype}

\usepackage[english]{babel}
\usepackage{csquotes} 

\usepackage{setspace}
\usepackage{amsmath,amsthm,amsfonts,amssymb}
\usepackage{mathtools}

\usepackage[shortlabels, inline]{enumitem}

\usepackage{xcolor}

\usepackage{xspace}
\usepackage{booktabs}
\usepackage{array}

\usepackage[hypertexnames=false]{hyperref}

\definecolor{darkblue}{rgb}{0,0,0.5}
\definecolor{cerule}{RGB}{53,122,183}
\definecolor{cardinal}{RGB}{184,32,16}
\definecolor{lightBlue}{rgb}{0,0.3,1}

\hypersetup{
    colorlinks,
    linkcolor={black},
    citecolor={cerule},
    urlcolor={cardinal}
}

\usepackage{float}
\usepackage[]{algpseudocode}
\algrenewcommand\textproc{}

\usepackage{tikz}
\usetikzlibrary{decorations,arrows,automata,trees,fit,shapes,intersections,positioning}
\definecolor{newLightBrown}{RGB}{230,51,18}

\usepackage[ruled]{caption}      
\usepackage{subcaption}

\floatstyle{ruled}
\newfloat{algo}{tp}{lop}
\floatname{algo}{Algorithm}
\allowdisplaybreaks[2]

\usepackage[
  citestyle=authoryear-comp,
  bibstyle=authoryear,
  backend=bibtex,
  isbn=false,
  maxcitenames=3,
  maxbibnames=6,
  giveninits=true,
  sortcites=false,
  url=false
]{biblatex}
\bibliography{maths}

\AtEveryBibitem{\clearfield{month}}
\AtEveryBibitem{\clearfield{day}}

\DeclareNameAlias{sortname}{first-last}

\def\bC{{\mathbb{C}}}
\def\bE{{\mathbb{E}}}

\def\bP{{\mathbb{P}}}

\def\bS{{\mathbb{S}}}

\def\cH{{\mathcal{H}}}

\def\cN{{\mathcal{N}}}

\def\cU{{\mathcal{U}}}
\def\cV{{\mathcal{V}}}

\def\bfu{{\mathbf u}}
\def\bfw{{\mathbf w}}
\def\bfv{{\mathbf v}}

\def\bfr{{\mathbf r}}

\def\Frob{{\mathrm{Frob}}}      

\def\Proj{\bP^n}

\def\EE{\mathbb{E}}

\def\FAIL{\textsc{Fail}}
\def\vol{\operatorname{vol}}

\def\mop{\operatorname}
\newcommand\smashintlong[2][]{\int_{\crampedrlap{#2}}\mathrlap{#1} {}_{\phantom{#2}}}
\newcommand\smashint[2][]{\int_{\crampedrlap{#2}}{#1}}

\def\leq{\leqslant}
\def\geq{\geqslant}

\def\det{\mop{det}\nolimits}

\def\epsilon{\varepsilon}
\def\phi{\varphi}

\def\tr{\mop{tr}}

\def\tn#1{\left\VERT #1 \right\VERT}

\newcommand{\abs}[1]{\left| #1 \right|}

\def\ud{\mathrm{d}}

\def\Kmax{K_{\mathrm{max}}}

\def\poly{\operatorname{poly}}

\def\ABP{algebraic branching program\xspace}
\def\GRABP{Gaussian random ABP\xspace}
\newcommand{\VBP}{\mathrm{VBP}}
\newcommand{\VP}{\mathrm{VP}}
\newcommand{\VNP}{\mathrm{VNP}}

\def\cbbbs{\text{cost}_\mathrm{BBBS}}

\def\eqdef{\stackrel{.}{=}}
\newcommand{\st}{\mathrel{}\middle|\mathrel{}}

\newtheorem{theorem}{Theorem}[section]
\newtheorem{proposition}[theorem]{Proposition}
\newtheorem{lemma}[theorem]{Lemma}
\newtheorem{corollary}[theorem]{Corollary}

\def\eproof{{\mbox{}\hfill\qed}\medskip}

\theoremstyle{definition}

\theoremstyle{remark}

\newtheorem{remark}[theorem]{Remark}

\numberwithin{equation}{section}

\title[Rigid continuation paths II]
{Rigid continuation paths\\
II. Structured polynomial systems}

\author[P. Bürgisser]{Peter Bürgisser}
\address{Institut f\"{u}r Mathematik, Technische Universit\"{a}t Berlin, Germany}
\email{pbuerg@math.tu-berlin.de}

\author[F. Cucker]{Felipe Cucker}
\address{Department of Mathematics, City University of Hong Kong, Hong Kong}
\email{macucker@cityu.edu.hk}

\author[P. Lairez]{Pierre Lairez}
\address{Inria, France}
\email{pierre.lairez@inria.fr}

\thanks{Supported by
  the ERC under the European Union's Horizon 2020 research and innovation 
programme (grant agreement no. 787840);
a GRF grant from the Research Grants Council of the Hong Kong SAR
(project number CityU~11300220);
and the projet \emph{De Rerum Natura} ANR-19-CE40-0018 of the French National
Research Agency (ANR)}

\date{October 19, 2020}

\subjclass[2000]{Primary 68Q25; Secondary 65H10, 65H20, 65Y20}

\begin{document}
\maketitle

\begin{abstract}
This work studies the average complexity of solving structured polynomial systems
that are characterized by a low evaluation cost, as opposed to 
the dense random model previously used.
Firstly, we design a continuation algorithm that computes, with high probability,
an approximate zero of a polynomial system given only as black-box evaluation
program.
Secondly, we introduce a universal model of random polynomial systems with
prescribed evaluation complexity~$L$.
Combining both, we show that we can compute an approximate zero of a
random structured polynomial system with $n$~equations of degree
at most~$\delta$ in $n$~variables with only
$\poly(n, \delta) L$ operations with high probability.
This exceeds the expectations implicit in Smale's 17th problem.
\end{abstract}

\tableofcontents

\section{Introduction}
\label{sec:introduction}

Can we solve polynomial systems in polynomial time?  This question
received different answers in different contexts.  The NP-completeness
of deciding the feasibility of a general polynomial system in both
Turing and BSS models of computation is certainly an important
difficulty but it does not preclude efficient algorithms for
computing all the zeros of a polynomial system or solving 
polynomial systems with as many equations as variables, for which the
feasibility over algebraically closed fields is granted under
genericity hypotheses.  And indeed, there are several ways of
computing all~$\delta^n$ zeros of a generic polynomial system of~$n$
equations of degree~$\delta > 1$ in~$n$ variables with
$\poly(\delta^n)$ arithmetic operations
\parencite[e.g.][]{Renegar_1989,Lakshman_1991,GiustiLecerfSalvy_2001}.

Smale's 17th problem~\parencite{Smale_1998} is a clear-cut formulation
of the problem in a numerical setting.  It asks for an algorithm, with
polynomial average complexity, for computing {\em one} approximate zero of a
given polynomial system, where the complexity is to be measured with
respect to the \emph{dense input size $N$}, that is, the number of
possible monomials in the input system.
Smale's question was given recently a positive answer
after seminal work by
\textcite{ShubSmale_1993b,ShubSmale_1993a,ShubSmale_1993,ShubSmale_1996,ShubSmale_1994},
fundamental contributions by
\textcite{BeltranPardo_2009,BeltranPardo_2011,Shub_2009}, as well as
our work \parencite{BurgisserCucker_2011, Lairez_2017}.  The basic
algorithmic idea underlying all these is \emph{continuation along
  linear paths}.
To find a zero of a system $F=(f_1,\ldots,f_n)$ of $n$ polynomials in
$n$ variables of degree at most~$\delta$, we first construct another
system~$G$ with a built-in zero~$\zeta_0 \in \mathbb{C}^n$ and consider
the family $F_t \eqdef tF+(1-t)G$ of polynomial systems.  If~$G$ is
generic enough, the zero~$\zeta_0$ of~$G$ extends as a continuous
family~$(\zeta_t)$ with~$F_t(\zeta_t)=0$, so that~$\zeta_1$ is a zero
of~$F$.  It is possible to compute an approximation of~$\zeta_1$ by
tracking $\zeta_t$ in finitely many steps.  From the perspective of
complexity analysis, the focal points are the choice of~$(G,\zeta_0)$
and the estimation of the number of steps necessary for a correct
approximation of~$\zeta_1$ (the cost of each step being 
not an issue as it is $O(N)$).
The problem of choosing an initial pair~$(G,\zeta_0)$ was for a long
while a major obstable in complexity analysis.
It was solved by \textcite{BeltranPardo_2009} who introduced an
algorithm to sample a random polynomial system~$G$ together with a
zero~$\zeta_0$ of it and provided a $\poly(n,\delta)N^2$ bound for
the average number of steps in the
numerical continuation starting from~$(G,\zeta_0)$.
This idea was followed in subsequent works with occasional cost improvements
that decreased the exponent in $N$ for the average number of steps.
Note that for a system
of~$n$ polynomial equations of degree~$\delta$ in~$n$ variables,
$N =n \binom{\delta+n}{n}$, 
and therefore $N \geq 2^{\min(\delta, n)}$.
Regarding Smale's question, an $N^{O(1)}$ bound on this number
is satisfactory but the question was posed, how much can
  the exponent in this bound be reduced?

\subsection{Rigid continuation paths}

The first part of this work \parencite{Lairez_2020}\footnote{Hereafter
 refered to as ``Part~I''.} gave an answer.
It introduced continuation along rigid paths: 
the systems $F_t$ have the form $F_t \eqdef \left(f_1 \circ
u_{1}(t),\ldots,f_n \circ u_{n}(t)\right)$ where the $u_{i}(t)\in
U(n+1)$ are unitary matrices that depend continuously on the
parameter~$t$, while~$f_1,\dotsc,f_n$ are fixed homogeneous polynomials.  

Compared to the previous setting, the natural parameter space for the continuation
is not anymore the full space of all polynomial
systems of a given degree, but rather
the group~$U(n+1)^n$, denoted~$\cU$.  We developed
analogues of Beltrán and Pardo's results for rigid paths.  Building on
this, we could prove a $\poly(n,\delta)$ bound on the average number
of continuation steps required to compute one zero of a Kostlan random
polynomial system,\footnote{A Kostlan random polynomial system is a
dense polynomial system where all coefficients are independent
Gaussian complex random variables with an appropriate scaling,
see~\S I.4.}
yielding a $N^{1+o(1)}$ total complexity bound.
This is the culmination of several results in this
direction which improved the average analysis number of continuation
steps (see Table~\ref{fig:prevwork}) for solving random dense polynomial systems.

\begin{table}[t]\centering\renewcommand{\arraystretch}{1.1}
\begin{tabular*}{\textwidth}{@{\extracolsep{\fill}} lllll}
  \toprule                                      & distribution      & $\EE[\#\text{steps}]$          & $\EE[\text{total cost}]$     \\\midrule
  \textcite{ShubSmale_1996}                     & Kostlan           & essentially $\poly(\delta^n)$        & \emph{not effective}         \\
  \textcite{ShubSmale_1994}                     & Kostlan           & $\poly(n,\delta) {N^3}$          & \emph{not effective}         \\
  \textcite{BeltranPardo_2009}                  & Kostlan           & $\poly(n,\delta) N^2$          & $\poly(n,\delta)N^3$         \\
  \textcite{BeltranShub_2009}                   & Kostlan           & $\poly(n, \delta)$             & \emph{not effective}         \\
  \textcite{BeltranPardo_2011}                  & Kostlan           & $\poly(n, \delta) N$           & $\poly(n,\delta)N^2$         \\
  \textcite{BurgisserCucker_2011}               & noncentered       & $\poly(n, \delta) N / \sigma$  & $\poly(n,\delta)N^2/\sigma$  \\
  \textcite{ArmentanoBeltranBurgisserEtAl_2016} & Kostlan           & $\poly(n, \delta) {N}^{\frac12}$    & $\poly(n,\delta)N^{\frac32} $\\
  \textcite{Lairez_2020}                        & Kostlan           & $\poly(n, \delta)$             & $\poly(n,\delta) N$          \\
  \bottomrule
\end{tabular*}
\medskip

\caption[]{
Comparison of previous complexity analysis of numerical continuation algorithms
for solving systems of~$n$ polynomial equations of degree~$\delta$ in~$n$ 
variables.
The parameter $N = n \binom{n+\delta}{n}$ is the dense input size.
The parameter $\sigma$ is the standard deviation for a noncentered distribution, in
the context of smoothed analysis.
Some results are not effective in that they do not lead to a complete algorithm to solve
polynomial systems.
}
\label{fig:prevwork}
\end{table}

\subsection{Refinement of Smale's question}\label{se:refinement} 

What is at stake beyond Smale's question, is the understanding of
numerical continuation as it happens in practice with a heuristic
computation of the step lengths.\footnote{To heuristically determine a
step length that is as large as possible, the principle is to try
some step length and check if Newton's iteration seems to
converge. Upon failure, the step length is reduced, and it is
increased otherwise. Of course, this may go wrong in many ways.}
Experiments have shown that certified algorithms in the Shub-Smale line
perform much smaller steps---and consequently many more steps---than
heuristic methods for numerical continuation
\parencite[]{BeltranLeykin_2012,BeltranLeykin_2013}.
In spite of progress in designing better and better
heuristics \parencite[e.g.,][]{Timme_2020,TelenVanBarelVerschelde_2020},
the design of efficient algorithms for
certified numerical continuation remains an important aspiration.  With a
view on closing the gap between rigorous step-length estimates and
heuristics, a first observation---demonstrated experimentally by
\textcite{HauensteinLiddell_2016} and confirmed theoretically in
Part~I---highlights the role of higher-order derivatives. Shub and
Smale's first-order step-length computation seems powerless in
obtaining $\poly(n, \delta)$ bounds on the number of steps: we need to
get closer to Smale's $\gamma$ to compute adequate step lengths (see
Section~\ref{sec:numer-cont-with} for a more detailed discussion).

However, estimating the higher-order derivatives occurring in $\gamma$ is
expensive. Thus, while using $\gamma$ improves the average number of steps, it
introduces a vice in the step-length computation. In Part~I, we obtained a
$\poly(n, \delta) N$ complexity bound for estimating the
variant~$\hat\gamma_\Frob$ of~$\gamma$ (Proposition~I.32) which, we showed, can
be used to estimate step lengths. This cost is quasilinear with respect to the
input size, we can hardly do better. But is~$N$ the right parameter to measure
complexity? From a practical point of view, $N$ is not so much relevant.
Often~$N$ is much larger than the number of coefficients that actually define
the input system, for example when the system is sparse or structured. This
observation is turned to practical account by treating the input system not as a
linear combination of monomials but as a black-box evaluation function, that is,
as a routine that computes the value of the components of the system at any
given point. Most implementations of numerical continuation do this. In this
perspective, $N$ does not play any role, and there is a need for adapting 
the computation of~$\gamma$.

Designing algorithms for black-box inputs and analyzing their complexity for
dense Gaussian random polynomial systems is interesting but misses an important
point. The evaluation complexity of a random dense polynomial system
is~$\Theta(N)$, whereas the benefit of considering a black-box input is
precisely to investigate systems with much lower evaluation complexity, and such
systems have measure zero in the space of all polynomial systems. It is
conceivable, even from the restricted perspective of numerical polynomial 
system
solving, that intrinsically, polynomial systems with low evaluation complexity
behave in a different way than random dense polynomial systems. So Smale's
original question of solving polynomial systems in polynomial time leads to the
following refined question:
\begin{quote}
  Can we compute an approximate zero of a {\em structured} polynomial system~$F$
  given by {\em black-box evaluation functions} with $\poly(n, \delta)$ many
  arithmetic operations and evaluations of~$F$ on average?
\end{quote}
We use algebraic branching programs (ABPs), a widely studied concept in
algebraic complexity theory (see \S\ref{se:context-ABP}), as a model of
computation for polynomials with low evaluation complexity. Further, we
introduce a natural model of {\em Gaussian random algebraic branching programs}
in order to capture the aspect of randomization. The main result of this paper
is an affirmative answer to the above refined question in this model.

\subsection{Polynomial systems given by black-box evaluation}
\label{sec:BB}

The model of computation is the BSS model, extended with rational exponentiation
for convenience and a ``6th type of node'', as introduced by
\textcite{ShubSmale_1996}, that computes an exact zero in $\bP^1$ of a bivariate
homogeneous polynomial given an approximate zero (this is used in the sampling
step), see Part~I, \S4.3.1 for a discussion. The term ``black-box'' refers to a
mode of computation with polynomials where we assume only the ability to
evaluate them at a complex point. Concretely, the polynomials are represented by
programs, or BSS machines. For a black-box
polynomial~$f \in\mathbb{C}[z_1,\dotsc,z_n]$, we denote by~$L(f)$ the number of
operations performed by the program representing~$f$ to evaluate~$f$ at a 
point
in~$\mathbb{C}^n$. For a
polynomial system~$F = (f_1,\dotsc,f_n)$, we
write~$L(F) \eqdef L(f_1)+\dotsb + L(f_n)$. It is possible that evaluating~$F$
costs less than evaluating its components separately, as some computations may
be shared, but we cannot save
more than a factor~$n$, so we ignore the issue.
More generally, in this article, we will not enter the details of the
$\poly(n,\delta)$ factors.
The ability to evaluate
first-order derivatives will also be used. For a univariate polynomial~$f$ of
degree {at most $\delta$} the derivative at~$0$ can be computed from evaluations
using the formula
\begin{equation}
  f'(0) = \frac1{\delta+1} \sum_{i=0}^{\delta} \omega^{-i} f(\omega^{i}),
\end{equation}
where~$\omega \in \mathbb{C}$ is a primitive~$(\delta+1)$th root of unity.
Similar formulas hold for multivariate polynomials. In practice, automatic
differentiation \parencite[e.g.,][]{BaurStrassen_1983} may be used. In any case,
we can evaluate the Jacobian matrix of a black-box polynomial system~$F$
with~$\poly(n,\delta) L(F)$ operations. Since this is below the resolution that
we chose, we do not make specific assumptions on the evaluation complexity of
the Jacobian matrix. Moreover, the degree of a black-box polynomial can be
computed with probability~1 in the BSS model by evaluation and interpolation
along a line.\footnote{If the values of a univariate polynomial~$f$ at $d+2$
  independent Gaussian random points coincide with the values of a degree 
at
  most~$d$ polynomial at the same points, then $f$ has degree at most~$d$ 
with
  probability~1, so we can compute, in the BSS model, the degree of a black-box
  univariate polynomial. Furthermore, the degree of a multivariate
  polynomial~$F$ is equal to the degree of the univariate polynomial obtained by
  restricting~$F$ on a uniformly distributed line passing through the origin,
  with probability~1.} So there is no need for the degree to be specified
separately.

\subsection{The $\Gamma(f)$ number}

Beyond the evaluation complexity~$L(F)$, the hardness of computing a zero 
of~$F$
in our setting depends on an averaged $\gamma$~number. For a
polynomial~$f \in \mathbb{C}[z_0,\dotsc,z_n]$, recall that
\begin{equation}\label{eq:34}
  \gamma(f, z) \eqdef \sup_{k\geq2} \left( \|\ud_z f\|^{-1}
  \tn{\tfrac{1}{k!} \ud_z^k f }\right)^{\frac{1}{k-1}},
\end{equation}
where the triple norm $\tn{A}$ of a $k$-multilinear map~$A$ is defined as
$\sup \frac{\|A(z_1,\dotsc,z_k)\|}{\|z_1\|\dotsb\|z_k\|}$. If~$f$ is homogeneous
and~$[z] \in \mathbb{P}^n$ is a projective point, we
define~$\gamma(f, [z]) \eqdef \gamma(f, z)$, for some
representative~$z\in\mathbb{S}(\mathbb{C}^{n+1})$. The definition does not
depend on the representative. By Lemma I.11 (in Part~I),
$\gamma(f, z) \ge \frac12 (\delta-1)$ if $f$ is homogeneous of degree $\delta$,
and $\gamma(f, z) =0$ if $\delta=1$. For computational purposes, we prefer the
\emph{Frobenius $\gamma$ number} introduced in Part~I:
\begin{equation}
  \label{eq:4}
  \gamma_\Frob(f,z) \eqdef
  \sup_{k\geq 2} \left(\left\| \ud_zf \right\|^{-1} \left\| \tfrac{1}{k!} 
\ud_z^kf
    \right\|_\Frob\right)^{\frac{1}{k-1}},
\end{equation}
where~$\|-\|_\Frob$ is the Frobenius norm of a multilinear map 
(see~\S I.4.2).  The two variants are tightly related (Lemma~I.29):
\begin{equation}\label{eq:67}
  \gamma(f, z) \leq \gamma_\Frob(f, z) \leq (n+1)\gamma(f, z).
\end{equation}
We will not
need here to define, or use, the $\gamma$~number of a polynomial system.
For a homogeneous polynomial~$f \in \mathbb{C}[z_0,\dotsc,z_n]$
of degree $\delta\ge 2$, we define
the \emph{averaged $\gamma$ number} as 
\begin{equation}\label{eq:defGamma}
  \Gamma(f) \eqdef \EE_\zeta \left[ \gamma_\Frob(f, \zeta)^2 \right]^{\frac12} {\in [\tfrac12,\infty]},
\end{equation}
where~$\zeta$ is a uniformly distributed zero of~$f$ in~$\mathbb{P}^n$. 
For a homogeneous polynomial system~$F = (f_1,\dotsc,f_n)$, we define
\begin{equation}\label{eq:defGammaF}
  \Gamma(F) \eqdef \left( \Gamma(f_1)^2 +\dotsb + \Gamma(f_n)^2 \right)^{\frac12}
  \leq \sum_{i=1}^n \Gamma(f_i).
\end{equation}
While~$L(F)$ reflects an algebraic structure, $\Gamma(F)$ reflects a numerical aspect.
In the generic case, where all the $f_i$ have only regular zeros,
$\Gamma(F)$ is finite (see Remark~\ref{rem:gamma-finiteness}).

Let~$d_1,\dotsc,d_n$ be integers~$\geq 2$ and let~$\cH$ be the space of
homogeneous polynomial systems~$(f_1,\dotsc,f_n)$
with~$f_i \in \mathbb{C}[z_0,\dotsc,z_n]$ homogeneous of degree~$d_i$.
Let $\delta \eqdef  \max_i d_i$.
Let~$\cU$ be the group~$U(n+1)^n$ made of~$n$ copies
of the group of unitary matrices of size~$n+1$.
For~$\bfu = (u_1,\dotsc,u_n) \in \cU$ and~$F = (f_1,\dotsc,f_n) \in \cH$, we
define the action
\begin{equation}\label{eq:48}
  \bfu \cdot F = \left( f_1 \circ u_1^{-1},\dotsc,f_n \circ u_n^{-1} \right).
\end{equation}
It plays a major role in the setting of rigid continuation paths. Note
that~$\Gamma$ is unitary invariant: $\Gamma(\bfu\cdot F)=\Gamma(F)$ for
any~$\bfu \in \cU$. Concerning $L$, we have
$L(\bfu \cdot F) \leq L(F) + O(n^3)$, using \eqref{eq:48} as a formula to
evaluate~$\bfu \cdot F$. (Note that the matrices $u_i$ are unitary, so the
inverse is simply the Hermitian transpose.)

\subsection{Main results I}

In our first main result, we design the {randomized}
algorithm \textsc{BoostBlackBoxSolve} in
the setting of rigid continuation paths (see~\S\ref{sec:valid-proof-theor}) for
computing with high probability an approximate zero of a black-box polynomial
system~$F$. We give an average analysis when the input system is~$\bfu \cdot F$
where~$\bfu$ is uniformly distributed and~$F$ is fixed.

\begin{theorem}[Termination and correctness]\label{thm:first-main-result-corretness}
  Let~$F = (f_1,\dotsc,f_n)$ be a homogeneous polynomial system with only regular zeros.
  On input~$F$, given as a black-box evaluation program, and~$\epsilon > 0$,
  Algorithm \textsc{BoostBlackBoxSolve} terminates almost surely and
  computes a point~$z\in \mathbb{P}^n$, 
  which is an approximate zero of~$F$ with 
probability at least~$1-\epsilon$.
\end{theorem}

Algorithm \textsc{BoostBlackBoxSolve} is a randomized algorithm of Monte Carlo type: the output is only correct with given probability.
Nonetheless, this probability is bounded below independently of~$F$.
This is in strong contrast with, for example, the main result of \textcite{ShubSmale_1996}, where the probability of success is relative to the input: on some inputs, the algorithm always succeeds, while on some others, 
it always fails.
The algorithm presented here succeeds with probability at least~$1-\epsilon$ for any
input with regular zeros (it may not terminate if this regularity hypothesis is
not satisfied).

Let~$\cbbbs(F, \epsilon)$ be the number of operations performed by
Algorithm~\textsc{BoostBlackBoxSolve}
on input~$F$ and~$\epsilon$. This is a random variable because the algorithm is randomized.
As in many previous works, we are unable to bound precisely~$\cbbbs(F, \epsilon)$,
or its expectation over the internal randomization for a fixed~$F$.
Instead, we introduce a randomization of the input polynomial system~$F$.
So we consider a random input~$H$ and study instead the expectation
of~$\cbbbs(H, \epsilon)$ over both~$H$ and the internal randomization.
The randomization that we introduce is a random unitary change of variable on
each equation of~$F$.
In other words, we consider~$H = \bfu\cdot F$ for a random uniformly
distributed~$\bfu \in \cU$.

We say that a system~$F=(f_1,\dotsc,f_n)$ is \emph{square-free} if~$f_1,\dotsc,f_n$
are square-free polynomials.
Importantly, this implies that with probability~1, 
$\bfu\cdot F$ has only regular zeros when~$\bfu \in \cU$ is uniformly distributed.

\begin{theorem}[Complexity analysis]\label{thm:first-main-result-complexity}
  Let~$F=(f_1,\dotsc,f_n)$ be a square-free homogeneous polynomial
  system with degrees at most~$\delta$ in~$n+1$ variables.
  Let~$\bfu \in \cU$ be uniformly distributed, and let~$H = \bfu \cdot F$.
  On input~$H$, given as a black-box evaluation program, and~$\epsilon>0$,
  Algorithm~\textsc{BoostBlackBoxSolve} terminates after
  \[
\mop{poly}(n, \delta) \cdot L(F) \cdot \left( \Gamma(F) \log \Gamma(F)
+ \log \log \epsilon^{-1} \right)
\]
operations on average.
   ``On average'' refers to expectation with respect to both
the random draws made by the algorithm and the random variable $\bfu$, but~$F$ is fixed.
\end{theorem}


In addition to the foundations laid in Part~I, the main underlying tool is a
Monte-Carlo method for estimating Smale's $\gamma$ number:
with~$\poly(n, \delta) \log \frac1\epsilon$ evaluations of~$f$, we can
estimate~$\gamma(f, z)$ within a factor~$\poly(n, \delta)$ with probability at
least~$1-\epsilon$ (see Theorem~\ref{thm:probabilistic-gamma}). This turns both
the computation of the step length and the whole zero-finding process into
Monte-Carlo algorithms themselves and, as a consequence, {\sc
  BoostBlackBoxSolve} departs from the simple structure of continuation
algorithms described above. During execution, {\sc BoostBlackBoxSolve} draws
real numbers from the standard Gaussian distribution to compute the initial pair
$(G,\zeta)$ and estimate various~$\gamma_\Frob$. The average cost in
Theorem~\ref{thm:first-main-result-complexity} is considered with respect 
to both this
inner randomization of the algorithm, and the randomness of the input $\bfu$ (or
$F$ in Corollary~\ref{coro:main-result} below).

{\sc BoostBlackBoxSolve} actually performs the continuation procedure several
times, possibly with different initial pairs, as well as a validation routine
that drastically decreases the probability that the returned point is not 
an
approximate zero of $F$. Its complexity analysis reflects this more complicated
structure.

In contrast with many previous work, {\sc BoostBlackBoxSolve} does not always
succeed: its result can be wrong with a small given probability~$\epsilon$, but
the doubly logarithmic dependence of the complexity with respect to~$\epsilon$
is satisfactory. We do not know if it is optimal but it seems difficult, in the
black-box model, to obtain an algorithm with similar complexity bounds but that
succeeds (i.e., returns a certified approximate zero) with probability~one: to
the best of our knowledge all algorithms for certifying zeros need some global
information---be it the Weyl norm of the system
\parencite{HauensteinSottile_2012} or evaluation in interval arithmetic
\parencite{RumpGraillat_2010}---which we cannot estimate with probability~1 in
the black-box model with only $\poly(n,\delta)$ evaluations. So unless we 
add an
\emph{ad hoc} hypothesis (such as a bound on the coefficients in the monomial
basis), we do not know how to certify an approximate zero in the black-box
model.

Theorem~\ref{thm:first-main-result-complexity} can be interpreted as an average
analysis on an orbit of the action of~$\cU$ on~$\cH$. More generally, we
may assume a random input~$F \in \cH$ where the distribution of~$F$ is 
\emph{unitary invariant}, meaning that for any~$\bfu\in\cU$, $\bfu \cdot F$ 
and~$F$ have the same distribution. This leads to the following statement.

\begin{corollary}\label{coro:main-result}
  Let~$F \in \cH$ be a random polynomial system,
  which is almost surely square-free,
  with unitary invariant distribution. 
  Let~$L$ be an upper bound on~$L(F)$
  and put~$\Gamma = \EE [ \Gamma(F)^2]^{\frac12}$. On input~$F$ (given as a black-box
  evaluation program) and~$\epsilon > 0$, Algorithm~{\sc BoostBlackBoxSolve}
  terminates after
  \[
    \mop{poly}(n, \delta) \cdot L \cdot \left( \Gamma \log \Gamma
      + \log \log \epsilon^{-1} \right)
  \]
  operations on average.
\end{corollary}

The quantity $\Gamma(F)$ strongly influences the average complexity {in}
Theorem~\ref{thm:first-main-result-complexity} and Corollary~\ref{coro:main-result}
and while it is natural to expect the complexity to depend on numerical aspects of
$F$, it is desirable to quantify this dependence 
by studying random distributions of~$F$ \parencite{Smale_1997}.
It was shown in Part~I that if~$F\in \cH$ is a
Kostlan random polynomial system, then~$\EE[ \Gamma(F)^2 ] = \poly(n, \delta)$
(Lemma~I.38). Together with the standard bound $L(F)=O(N)$, we {immediately
  obtain from Corollary~\ref{coro:main-result}} the following complexity
analysis, similar to the main result of Part~I (Theorem~I.40), but assuming only
a black-box representation of the input polynomial system.

\begin{corollary}\label{coro:complexity-gaussian-dense}
  Let~$F \in \cH$ be a Kostlan random polynomial system. On input~$F$
  and~$\epsilon > 0$, Algorithm~{\sc BoostBlackBoxSolve} terminates after
  $\mop{poly}(n, \delta) \cdot \log \log \epsilon^{-1}$ operations and evaluations of~$F$ on
  average.
\end{corollary}

Our second main result (Theorem~\ref{thm:second-main} below) states that
exact same bound for polynomial system given by independent Gaussian random algebraic
branching programs. We next introduce this model.

\subsection{Algebraic branching programs}\label{se:ABP}

Following \textcite{Nisan_1991}, an \emph{algebraic branching program} (ABP) of
degree~$\delta$ is a labeled directed acyclic graph with one source and one
sink, with a partition of the vertices into levels, numbered from~0 to~$\delta$,
such that each edge goes from level~$i$ to level~$i+1$. The source is the 
only
vertex at level~0 and the sink is the only vertex at level~$\delta$. Each 
edge
is labeled with a homogeneous linear form in the input
variables~$z_0,\dotsc,z_n$. An ABP \emph{computes} the polynomial obtained as
the sum over all paths from the source to the sink of the product of the linear
forms by which the edges of the path are labelled. It is a
homogeneous polynomial of degree~$\delta$. The {\em width}~$r$ of the ABP 
is
the maximum of the cardinalities of the level sets. The {\em
  size} $s$ of the ABP, which is defined as the number of its vertices,
satisfies $r \le s \le (\delta-1)r +2$. Any homogeneous polynomial~$f$ can be
computed by an ABP and the minimum size or width of an ABP computing~$f$ are
important measures of the complexity of~$f$, see \S\ref{se:context-ABP}.

While ABPs provide an elegant graphical way of formalizing computations with
polynomials, we will use an equivalent matrix formulation. Suppose that the
$i$th level set has $r_i$ vertices and let $A_i(z)$ denote the weighted
adjacency matrix of format $r_{i-1}\times r_i$, whose entries are the weights of
the edges between vertices of level~$i-1$ and level~$i$. Thus the entries 
of
$A_i(z)$ are linear forms in the variables $z_0,\ldots,z_n$. The polynomial
$f(z)$ computed by the ABP can then be expressed as the trace of iterated
  matrix multiplication, namely,
\begin{equation}\label{eq:fnABP} 
  f(z) =  \tr \left( A_{1}(z) \dotsb A_{\delta}(z)\right) .
\end{equation}
It is convenient to relax the assumption $r_0=r_\delta=1$ to $r_0=r_\delta$.
Compared to the description in terms of ABPs, this adds some flexibility because
the trace is invariant under cyclic permutation of the matrices~$A_i(z)$.

Using the associativity of matrix multiplication, we can evaluate $f(z)$
efficiently by iterated matrix multiplication, which amounts to $O(\delta 
r^3)$
additions or multiplications of matrix entries; taking into account the cost
$O(n)$ of evaluating a matrix entry (which is a linear forms in the
variables~$z_0,\dotsc,z_n$), we see that we can evaluate~$f$ with a total 
of
$O(\delta r^2 n \delta r^3)$ arithmetic operations.

\subsection{Main results II}\label{se:MainResults-II}

Given positive integers $r_1,\dotsc,r_{\delta-1}$, we can form a random ABP
(that we call \emph{\GRABP}) of degree~$\delta$ by considering a directed
acyclic graph with~$r_i$ vertices in the layer~$i$
(for~$1\leq i \leq \delta-1$), one vertex in the layers~0 and~$\delta$, and all
possible edges from a layer to the next, labelled by linear forms in
$z_0\ldots,z_n$ with independent and identically distributed complex Gaussian
coefficients. This is equivalent to assuming that the adjacency matrices are
linear forms $A_{i}(z) = A_{i0} z_0 +\cdots + A_{in} z_n$ with independent
complex standard Gaussian matrices~$A_{ij} \in \mathbb{C}^{r_{i-1} \times 
r_i}$.

We call a \GRABP \emph{irreducible} if all layers (except the first and the
last) have at least two vertices. The polynomial computed by an irreducible
\GRABP is almost surely irreducible (Lemma~\ref{le:GRABP-irred}), and
conversely, the polynomial computed by a Gaussian random ABP that is not
irreducible is not irreducible; which justifies the naming.

Recall the numerical parameter $\Gamma$ {entering} the complexity of numerical
continuation in the rigid setting, see~\eqref{eq:defGamma} and
Theorem~\ref{thm:first-main-result-complexity}.
The second main result in this article is
{an upper bound on the expectation of}~$\Gamma(f)$, when~$f$ is computed by
a~\GRABP. Remarkably, the bound does not depend on the sizes~$r_i$ of the 
layers
defining the~\GRABP; in particular it is independent of its width!

\begin{theorem}\label{thm:second-main}
  If~$f$ is the random polynomial computed by an irreducible Gaussian random~ABP
  of degree~$\delta$, then
  \[
    \EE \left[ \Gamma(f)^2 \right] \leq \tfrac34 \delta^3 (\delta+n) \log 
\delta.
  \]
\end{theorem}

The distribution of the polynomial computed by a \GRABP is unitarily invariant
so, as a consequence of Corollary~\ref{coro:main-result}, we obtain polynomial
complexity bounds for solving polynomial systems made of \GRABP.

\begin{corollary}\label{coro:second-main}
  If~$f_1,\dotsc,f_n$ are independent irreducible Gaussian random
  ABPs of degree at most~$\delta$ and evaluation
  complexity at most~$L$, then~{\sc BoostBlackBoxSolve}, on
  input~$f_1,\dotsc,f_n$ and~$\epsilon > 0$,  
  terminates after
  \[ \poly(n,\delta) \cdot L \cdot \log\log \epsilon^{-1} \]
  operations on average.
\end{corollary}

This result provides an answer to the refined Smale's problem raised at the end
of \S\ref{se:refinement}, where ``structured'' is interpreted as ``low
evaluation complexity in the ABP model''.

The polynomial systems computed by ABPs of with~$r$ form a zero measure subset
of $\cH$ when~$n$ and~$\delta$ are large enough. More precisely, they form a subvariety
of~$\cH$ of dimension at most~$r^2 \delta n$ while the dimension of~$\cH$ 
grows
superpolynomially with~$n$ and~$\delta$. Note also that a polynomial~$f$
computed by a Gaussian random ABP may be almost surely singular (in the sense
that the projective hypersurface that it defines is singular), see Lemma~\ref{lem:singular-abp}.
This strongly contrasts with previously considered stochastic model of polynomial systems.

Lastly, it would be interesting to describe the limiting distribution of the
polynomial computed by a Gaussian random ABP as the size of the layers goes to infinity. 
Since this question is out of the scope of this article, we leave it open.

\subsection{On the role of algebraic branching programs}\label{se:context-ABP}

To motivate our choice of the model of ABPs, we point out here their important
role in algebraic complexity theory, notably in Valiant’s algebraic framework of
NP-completeness~\parencite{Valiant_1979,Valiant_1982}, see also
\textcite{Burgisser_2000}. This model features the complexity class~$\VBP$,
which models efficiently computable polynomials as sequences of multivariate
complex polynomials~$f_n$, where the degree of $f_n$ is polynomially bounded
in~$n$ and the homogeneization of $f_n$ can be computed by an ABP of width
polynomially bounded in~$n$. It is known \parencite{Toda_1992,
  MalodPortier_2008} that the sequence of determinants of generic $n\times n$
matrices is complete for the class~$\VBP$: this means the determinants have
efficient computations in this model and moreover, any $(f_n)\in \VBP$ can be
tightly reduced to a sequence of determinants in the sense that $f_n$ can 
be
written as the determinant of a matrix, whose entries are affine linear forms,
and such that the size of the matrix is polynomially bounded in $n$. The related
complexity class $\VP$ consists of the sequences of multivariate complex
polynomials $f_n$, such that the degree of $f_n$ grows at most polynomially
in~$n$ and such that $f_n$ can be computed by an arithmetic circuit
(equivalently, straightline program) of size polynomially bounded in $n$. 
While
it is clear that $\VBP\subseteq\VP$, it is a longstanding open question whether
equality holds. However, after relaxing ``polynomially bounded`'' to
``quasi-polynomially bounded''~\footnote{\emph{Quasi-polynomially bounded} in~$n$ means bounded by $2^{(\log n)^c}$ for some
  constant~$c$.}, the classes collapse (e.g., see \textcite{MalodPortier_2008}).
These results should make clear the relevance and universality of the model of
ABPs. Moreover, \textcite{Valiant_1979} defined another natural complexity class
$\VNP$, formalizing efficiently definable polynomials for which the sequence of
permanents of generic matrices is complete. Valiant's conjecture $\VBP\ne\VNP$
is a version of the famous $\mathrm{P}\ne\mathrm{NP}$ conjecture.

\subsection{Organization of paper}

In Section~\ref{sec:numer-cont-with} we first recall the basics of the
complexity analysis of numerical continuation algorithms and summarize the
results obtained in Part~I. Section~\ref{sec:black-box-evaluation} is devoted to
numerical continuation algorithms when the functions are given by a
black-box. 
We introduce here a sampling algorithm to estimate~$\gamma_\Frob$ with high
probability in this setting. Section~\ref{sec:complexity-analysis} is devoted to
the complexity analysis of the new algorithm on a random input~$\bfu\cdot 
F$. In
particular, in \S\ref{sec:valid-proof-theor}, we consider the problem of
certifying an approximate zero in the black-box model and we prove
Theorem~\ref{thm:first-main-result-complexity}. Finally,
Section~\ref{sec:gauss-algebr-branch} presents the proof of
Theorem~\ref{thm:second-main}, our second main result.

\section{Numerical continuation with few steps}
\label{sec:numer-cont-with}

\subsection{The classical setting}\label{sec:classical-setting}

Numerical continuation algorithms have been so far the main tool for the
complexity analysis of numerical solving of polynomial systems. We present here
the main line of the theory as developed by
\textcite{ShubSmale_1993b,ShubSmale_1993a,ShubSmale_1996,ShubSmale_1994,BeltranPardo_2009,BeltranPardo_2011,Beltran_2011}.
The general idea to solve a polynomial system~$F \in \mathcal{H}$ consists {of}
embedding~$F$ in a one-parameter continuous family~$(F_t)_{t\in[0,1]}$ of
polynomial systems such that~$F_1 = F$ and a zero of~$F_0$,
say~$\zeta_0\in \mathbb{P}^n$ is known. Then, starting from~$t=0$ and
$z=\zeta_0$, $t$ and~$z$ are updated to track a zero of~$F_t$ all along 
the path
from~$F_0$ to~$F_1$, as follows:
\begin{algorithmic}
  \State \algorithmicwhile\ {$t < 1$} \algorithmicdo\
    $t \gets t + \Delta t$ ;
   $z\gets \mop{Newton}(F_t, z)$\quad \algorithmicend\ \algorithmicwhile,
\end{algorithmic}
where~$\Delta t$ needs to be defined. The idea is that~$z$ always stays close
to~$\zeta_t$, the zero of~$F_t$ obtained by continuing~$\zeta_0$. To ensure
correctness, the increment~$\Delta t$ should be chosen small enough. But the
bigger~$\Delta t$ is, the fewer iterations will be necessary, meaning a better
complexity. The size of~$\Delta t$ is typically controlled with, on the one
hand, effective bounds on the variations of the zeros of~$F_t$ as $t$ changes,
and on the other hand, effective bounds on the convergence of Newton's
iteration. The general principle to determine~$\Delta t$ is the following, in
very rough terms because a precise argument generally involve lengthy
computations. The increment $\Delta t$ should be small enough so
that~$\zeta_{t}$ is in the basin of attraction around~$\zeta_{t+\Delta t}$ of
Newton's iteration for~$F_{t+\Delta t}$. This leads to the rule-of-thumb
$\| \Delta \zeta_t \| \rho(F_{t+\Delta t}, \zeta_{t+\Delta t}) \lesssim 1$,
where~$\Delta \zeta_t = \zeta_{t+\Delta t} - \zeta_{t}$ and~$\rho(F_t, \zeta_t)$
is the inverse of the radius of the basin of attraction of Newton's iteration. A
condition that we can rewrite as
\begin{equation}
  \frac{1}{\Delta t} \gtrsim \rho(F_t, \zeta_t)
  \left\| \frac{\Delta \zeta_t}{\Delta t} \right\|,
\end{equation}
assuming that
\begin{equation}
  \rho(F_{t+\Delta t}, \zeta_{t+\Delta t}) \simeq \rho(F_t, \zeta_t).
  \label{eq:42}
\end{equation}
The factor $\frac{\Delta \zeta_t}{\Delta t}$ is almost the
derivative~$\dot \zeta_t$ of~$\zeta_t$ with respect to~$t$.  It is
generally bounded using a \emph{condition number} $\mu(F_t, \zeta_t)$,
that is the largest variation of the zero~$\zeta_t$ after a
pertubation of~$F_t$ in~$\mathcal{H}$, so that
\begin{equation}\label{eq:40}
  \left\| \frac{\Delta \zeta_t}{\Delta t} \right\|
  \simeq \| \dot \zeta_t \| \leq \mu(F_t, \zeta_t) \|\dot F_t\|,
\end{equation}
where~$\dot F_t$ (resp.~$\dot\zeta_t$) is the derivative of~$F$ (resp.~$\zeta_t$) with respect to~$t$,
and the right-hand side is effectively computable.
The parameter~$\rho(F_t, \zeta_t)$ is much deeper.
Smale's $\alpha$-theory has been a preferred tool to
deal with it in many complexity analyses.
The number~$\gamma$ 
takes a prominent role in the theory
and controls the convergence of Newton's iteration \parencite{Smale_1986}:
$\rho(F_t, \zeta_t) \lesssim \gamma(F_t, \zeta_t)$.  
(For the definition of $\gamma(F,\zeta)$, e.g., see Eq.~(8) in Part~I.)
So we obtain the condition
\begin{equation}\label{eq:35}
  \frac{1}{\Delta t} \gtrsim \gamma(F_t, \zeta_t) \mu(F_t, \zeta_t) \|\dot F_t\|
\end{equation}
that ensures the correctness of the algorithm.
A rigorous argument requires a nice behavior of both
factors~$\gamma(F_t, \zeta_t)$ and~$\mu(F_t, \zeta_t)$ as $t$ varies,
this is a crucial point, especially in view of the assumption~\eqref{eq:42}.  The
factor~$\|\dot F_t\|$ is generally harmless; the factor~$\mu(F_t,
\zeta_t)$ is important but the variations with respect to~$t$ are
generally easy to handle; however the variations of~$\gamma(F_t,
\zeta_t)$ are more delicate.  This led \textcite{ShubSmale_1993b} to
consider the upper bound (called ``higher-derivative estimate'')
\begin{equation}\label{eq:39}
  \gamma(F, z) \lesssim \mu(F, z),
\end{equation}
with the same~$\mu$ as above,
and the subsequent correctness condition
\begin{equation}\label{eq:38}
  \frac{1}{\Delta t} \gtrsim \mu(F_t, \zeta_t)^2 \|\dot F_t\|.
\end{equation}
Choosing at each iteration $\Delta t$ to be the largest possible value
allowed by \eqref{eq:38},
we obtain a numerical continuation algorithm, with adaptive step length,
whose number~$K$ of iterations
is bounded, as shown first by \textcite{Shub_2009}, by
\begin{equation}\label{eq:36}
  K \lesssim \int_0^1 \mu(F_t, \zeta_t)^2 \|\dot F_t\| \ud t.
\end{equation}

It remains to choose the starting system~$F_0$, with a built-in zero~$\zeta_0$,
and the path from~$F_0$
to~$F_1$.  For complexity analyses, the most common choice of path is
a straight-line segment in the whole space of polynomial
systems~$\mathcal{H}$.
For the choice of the starting
system~$F_0$, \textcite{BeltranPardo_2009,BeltranPardo_2008} have
shown that a Kostlan random system is a relevant choice and that there is 
a simple
algorithm to sample a random system with a known zero.  If~$F_1$ is
also a random Gaussian system, then all the intermediate systems~$F_t$
are also random Gaussian, and using~\eqref{eq:36}, we obtain a bound,
following Beltrán and Pardo, on the expected number of iterations in
the numerical continuation from~$F_0$ to~$F_1$:
\begin{equation}\label{eq:37}
  \mathbb{E}_{F_0, \zeta_0, F_1}[K] \simeq \mathbb{E}_{F, \zeta}[\mu(F, \zeta)^2] \simeq \dim \mathcal{H},
\end{equation}
where~$\zeta$ is a random zero of~$F$.
The dimension of~$\mathcal{H}$ is the number of coefficients in~$F$,
it is the \emph{input size}.
For~$n$ equations of degree~$\delta$ in~$n$ variables, we compute
\begin{equation}
  \dim \mathcal{H} = n \binom{n+\delta}{n}.
\end{equation}
This is larger than any polynomial in~$n$ and~$\delta$ (as~$n$
and~$\delta$ go to~$\infty$),
but is much smaller
than~$\delta^n$, the generic number of solutions of such a system.
The cost of an iteration (computing the step size and performing one
Newton's iteration) is also bounded by the input size.  So we have an
algorithm whose average complexity is polynomial in the input
size. This is a major complexity result because it breaks the
$\mop{poly}(\delta^n)$ barrier set by algorithms that compute all
solutions simultaneously.  However, the bound~\eqref{eq:37} on the
expected number of iterations is still much larger than what heuristic
algorithms seem to achieve.

A first idea to design a faster algorithm would be to search for a
better continuation path in
order to lower the right-hand side in~\eqref{eq:36}.  Such paths do
exist and can give a $\mop{poly}(n,\delta)$ bound on~$\mathbb{E}[K]$
\parencite{BeltranShub_2009}.  Unfortunately, their computation
requires, in the current state of the art, to solve the target system
first.
A second approach focuses on sharpening the correctness
condition~\eqref{eq:38}, that is, on making bigger continuation steps.
The comparison of~\eqref{eq:38} with heuristics shows that there is room
for improvement \parencite{BeltranLeykin_2012,BeltranLeykin_2013}.
In devising this condition, two inequalities are too generous.  Firstly,
Inequality~\eqref{eq:40} bounds the variation of~$\zeta_t$ by the
worst-case variation.  The average worst-case variation can only grow
with the dimension of the parameter space, $\dim
\mathcal{H}$, and it turns out to be much bigger than the average
value of~$\|\dot \zeta_t\|$, which is~$\mop{poly}(n,\delta)$.  This
was successfully exploited
by~\textcite{ArmentanoBeltranBurgisserEtAl_2016} to obtain the bound
$\mathbb{E}[K] \lesssim \sqrt{\dim \mathcal{H}}$
for random Gaussian systems.
They used straight-line continuation paths but a finer computation
of the step size. The other inequality that turns out to be too coarse
is~\eqref{eq:39}: the higher derivatives need to be handled more accurately.

\subsection{Rigid continuation paths}

In Part~I, we introduced rigid continuation paths to obtain,
in the case of random Gaussian systems, the bound
\begin{equation}
  \mathbb{E}[K] \leq \mop{poly}(n, \delta).
\end{equation}

To solve a polynomial system~$F = (f_1,\dotsc,f_n) \in \mathcal{H}$
in~$n+1$ homogeneous variables,
we consider continuation paths having the form
\begin{equation}
  F_t \eqdef \left( f_1 \circ u_1^{-1}(t),\dotsc, f_n\circ u_n^{-1}(t) \right),
\end{equation}
where~$u_1(t),\dotsc,u_n(t) \in U(n+1)$ are unitary matrices depending
on the parameter~$t$, with~$u_i(1)=\mop{id}$.
The parameter space for the numerical continuation is not~$\mathcal{H}$
anymore but $U(n+1)^n$, denoted~$\cU$, a real manifold of dimension~$n^3$.
For~$\bfu = (u_1,\dotsc,u_n) \in \cU$ and~$F \in \mathcal{H}$, we denote
\begin{equation}
  \bfu \cdot F \eqdef \left( f_1 \circ u_1^{-1},\dotsc, f_n\circ u_n^{-1} 
 \right)
  \in \mathcal{H}.
\end{equation}
We developed in this setting an analogue of Beltrán and Pardo's
algorithm.  Firstly, we sample uniformly~$\bfv \in \cU$ together with
a zero of the polynomial system~$\bfv \cdot F$.  The same kind of
construction as in the Gaussian case makes it possible to perform this
operation without solving any polynomial system (only $n$ univariate
equations).  Then, we construct a path~$(\bfu_t)_{t\in [0,1]}$
in~$\cU$ between~$\bfv$ and {the unit~$\mathbf{1}_\cU$ in $\cU$},
and perform numerical
continuation using~$F_t \eqdef \bfu_t \cdot F$. The general strategy
sketched in \S\ref{sec:classical-setting} applies but the rigid setting features
important particularities.  The most salient of which is the average
conditioning, that is, the average worst-case variation of~$\zeta_t$
with respect to infinitesimal variations of $\bfu_t$.  It is now
$\mop{poly}(n)$ (see \S I.3.2), mostly because the dimension of the
parameter space is~$\mop{poly}(n)$.  Besides, the way the continuation
path is designed preserves the geometry of the equations.  This is
reflected in a better behavior of~$\gamma(F_t, \zeta_t)$ as~$t$ varies, 
which makes it possible to use an  upper bound much finer 
than~\eqref{eq:39}, that we called the \emph{split $\gamma$ number}.  In
the case of a random Gaussian input, we obtained in the end a
$\mop{poly}(n,\delta)$ bound on the average number of iterations for
performing numerical continuation along rigid paths.

\subsection{The split $\gamma$ number}
\label{sec:split-gamma-number}

Computing a good upper bound of the~$\gamma$ number is the key to make
bigger continuation steps.  We recall here the upper bound introduced
in Part~I.  The \emph{incidence condition number} of~$F =
(f_1,\dotsc,f_n)$ at $z$ is
\begin{equation}\label{eq:31}
  \kappa(F, z) \eqdef \tn{ \left(\ud_z F_z\right)^\dagger },
\end{equation}
where~$\dagger$ denotes the Moore--Penrose pseudoinverse and~$F_z$ the
normalized system
\begin{equation}
  F_z \eqdef \left( \frac{f_1}{\|\ud_z f_1\|},\dotsc, \frac{f_n}{\|\ud_z f_n\|} \right).
\end{equation}
When~$z$ is a zero of~$F$, this quantity depends only on the angles formed by
the tangent spaces at~$z$ of the $n$ hypersurfaces~$\left\{f_i=0 \right\}$ (see \S I.2.1 and \S I.3 for more details). It is closely related to the
intersection condition number introduced by \textcite{Burgisser_2017}. In 
the
context of rigid paths, it is also the natural condition number: the variation
of a zero~$\zeta$ of a polynomial system~$\bfu \cdot F$ under a perturbation
of~$\bfu$ is bounded by~$\kappa(\bfu \cdot F,\zeta)$ (Lemma I.16). Moreover, $F$
being fixed, if~$\bfu \in \cU$ is uniformly distributed and if~$\zeta$ is 
a
uniformly distributed zero of~$\bfu \cdot F$,
then~$\mathbb{E}[\kappa(\bfu\cdot F, \zeta)^2] \leq 6n^2$ (Proposition I.17).

The {\em split $\gamma$ number} is defined as
\begin{equation}\label{eq:41}
  \hat\gamma(F, z) \eqdef \kappa(F, z)
  \left( \gamma(f_1,z)^2 + \dotsb + \gamma(f_n, z)^2 \right)^{\frac12}.
\end{equation}
It tightly upper bounds~$\gamma(F, z)$ in that (Theorem I.13)
\begin{equation}
  \gamma(F, z) \leq \hat \gamma(F, z) \leq n \kappa(F, z) \gamma(F, z).
\end{equation}

\begin{algo}[tp]
  \centering
  \begin{algorithmic}
    \Function{{\sc NC}}{$F$, $\bfu$, $\bfv$, $z$}
    \State $(\bfw_t)_{0\leq t \leq T} \gets$ a 1-Lipschitz continuous path
    from~$\bfv$ to~$\bfu$ in~$\cU$
    \State{$t \gets 0$}
    \While{true}
    \For{$i$ from 1 to $n$}
    \State $w \gets$ $i$th component of~$\bfw_t$
    \State $g_i \gets \gamma_\Frob(f_i\circ w^{-1}, z)$ \Comment See~\eqref{eq:4}.
    \EndFor
    \State $t \gets t + \left(240 \, \kappa(\bfw_t, z)^2 \left(
        \sum_{i=1}^n g_i^2 \right)^{\frac12}\right)^{-1}$
    \Comment See~\eqref{eq:31} and~\eqref{eq:41F}.
    \If{$t \geq T$}
    \State \textbf{return} $z$
    \EndIf
    \State $z \gets \mop{Newton}({\bfw_t}\cdot F, z)$
    \Comment Newton iteration
    \EndWhile
    \EndFunction
  \end{algorithmic}
  \caption[]{Rigid numerical continuation, original version
    \begin{description}
      \item[Input:] $F \in\cH$, $\bfu$, $\bfv \in \cU$ and~$z \in \Proj$
      \item[Precondition:] $z$ is a zero of~$\bfv \cdot F$.
      \item[Output:] $w\in\bP^n$ if algorithm terminates.
      \item[Postcondition:] $w$ is an approximate zero of~$\bfu \cdot F$.
    \end{description}
  }
  \label{algo:nc}
\end{algo}

Whereas~$\gamma(F, z)$ does not behave nicely as a function of~$F$,
the split variant behaves well in the rigid setting: $F$ being fixed,
the function $\cU \times \mathbb{P}^n \to \mathbb{R}$, $(\bfu, z)
\mapsto \hat\gamma(\bfu \cdot F, z)^{-1}$ is 13-Lipschitz continuous
(Lemma I.21).\footnote{Note that the importance of such a Lipschitz
property has been highlighted by \textcite{Demmel_1987}. It implies
that~$1/13 \gamma$ is upper bounded on~$\cU\times \mathbb{P}^n$ by the
distance to the subset of all pairs~$(\bfu, \zeta)$ where~$\zeta$ is
a singular zero of~$\bfu\cdot F$.}
This makes it possible to
perform numerical continuation.  Note that we need not compute~$\gamma$
exactly, an estimate within a fixed ratio is enough.  For
computational purposes, we rather use the variant~$\gamma_\Frob$,
defined in~\eqref{eq:4}, in which the operator norm is replaced by a
Hermitian norm. It induces a {\em split $\gamma_\Frob$ number}
\begin{equation}\label{eq:41F}
  \hat\gamma_\Frob(F, z) \eqdef \kappa(F, z)
  \left( \gamma_\Frob(f_1,z)^2 + \dotsb + \gamma_\Frob(f_n, z)^2 \right)^{\frac12}
\end{equation}
as in~\eqref{eq:41}.
Algorithm~\ref{algo:nc} describes the computation of an
approximate zero of a polynomial system~$\bfu \cdot F$, given a zero
of some~$\bfv\cdot F$. (It is the same as Algorithm~I.2, 
with $\hat\gamma_\Frob$ for~$g$ and~$C=15$,
which gives the constant~240 that appears in Algorithm~\ref{algo:nc}.)
As an analogue of~\eqref{eq:36}, Theorem~I.23
bounds the number~$K$ of continuation steps performed by
Algorithm~\ref{algo:nc} as an integral over the continuation path:
\begin{equation}\label{eq:41U}
  K \leq 325 \int_0^T \kappa(\bfw_t\cdot F, \zeta_t)
  \hat \gamma_\Frob(\bfw_t\cdot F, \zeta_t) \ud t. 
\end{equation}
Based on this bound, we obtained in Part~I the following average analysis.


Let~$F = (f_1,\dotsc,f_n) \in \cH$ be a square-free polynomial system.
holds, for
Then, for uniformly random $\bfu,\bfv \in\cU$ and a uniformly random zero 
$z\in\Proj$ of $\bfv\cdot F$,
is the rigid solution variety of $F$; see I\S2 and \S4, 
Algorithm~\ref{algo:nc} terminates
almost surely and outputs an approximate zero of $\bfu\cdot F$. 
Moreover, the number~$K$ of continuation steps it performs 
satisfies $\bE[K] \leq 9000 n^3\,\Gamma(F)$
(Theorems~I.25 and~I.27, with~$\mathfrak g_i = {\gamma_\Frob}$ and~$C'=5$, according to
Lemma~I.31).
Here $\Gamma(F)$ denotes the crucial parameter introduced in~\eqref{eq:defGammaF}.

In case we cannot compute~$\gamma_\Frob$ exactly, but instead an
upper bound~$A$ such that~$\gamma_\Frob \leq A \leq M \gamma_\Frob$, for some
fixed~$M \geq 1$, the algorithm works as well, but the bound on the average
number of continuation steps is multiplied by~$M$ (see Remark~I.28):
\begin{equation}\label{eq:boundK}
  \bE[K] \leq  9000 n^3 M \, \Gamma(F).
\end{equation}

\begin{remark}\label{rem:gamma-finiteness}
  It is not completely clear for which systems $F$ does~$\Gamma(F)$ take
  finite or infinite value. Since~$\Gamma(F)^2=\Gamma(f_1)^2+\dotsb+\Gamma(f_n)^2$,
  it suffices to look only at~$\Gamma(f)$ for some homogeneous
  polynomial~$f\in \mathbb{C}[z_0,\dotsc,z_n]$.
Let~$X= \left\{ \zeta\in \mathbb{P}^n \st f(\zeta) = 0 \right\}$ and
$\Sigma = \left\{ \zeta\in X\st \ud_\zeta f = 0 \right\}$ be its
singular locus. If~$\Sigma = \varnothing$,
 then~$x\mapsto \gamma(f, x)$ is continuous, 
hence bounded on the compact
set~$X$ and it follows that~$\Gamma(f) < \infty$.
In the case where $\Sigma$ has codimension~1 or~0 in~$X$, we can show
that~$\Gamma(f) = \infty$.
But the general situation is not clear.
In particular, it would be interesting to interpret~$1/\Gamma(f)$ as
the distance to some set of polynomials.
\end{remark}

To obtain an interesting complexity result for a given class of
unitary invariant distributions  of polynomial systems~$F$,
based on numerical continuation along
rigid paths and Inequality~\eqref{eq:boundK}, we need, firstly, to
specify how to compute or approximate~$\gamma_\Frob$ at a reasonable
cost, and secondly, to estimate the expectation of $\Gamma(F)$
over $F$.
For the application to dense Gaussian systems, considered in Part~I, $\gamma_\Frob$ is
computed directly, using the monomial representation of the system to
compute all higher derivatives, and the estimation of~$\Gamma(F)$ is
mostly standard.  Using the monomial representation is not efficient
anymore in the black-box model.  We will rely instead on a
probabilistic estimation of~$\gamma_\Frob$, within a
factor~$\mop{poly}(n, \delta)$. However, this estimation may fail with
small probability, compromising the correctness of the result.

\section{Fast numerical continuation for black-box functions}
\label{sec:black-box-evaluation}

\subsection{Weyl norm}\label{se:weyl}
We recall here how to characterize the Weyl norm of a homogeneous polynomial as
an expectation, which is a key observation behind algorithm {\sc GammaProb} to
approximate~$\gamma_\Frob(f, z)$ by random sampling.

Let~$f \in \mathbb{C}[z_0,\dotsc,z_n]$ be a homogeneous polynomial of
degree~$\delta > 0$. In the monomial basis, $f$ decomposes
as~$\sum_\alpha c_\alpha z^\alpha$, where~$\alpha = (\alpha_0,\dotsc,\alpha_n)$
is a multi-index. The Weyl norm of~$f$ is defined as
\begin{equation}
  \|f\|_W^2 \eqdef \sum_\alpha \frac{\alpha_0! \dotsb \alpha_n!}{\delta!} 
\abs{c_\alpha}^2.
\end{equation}
The following statement seems to be classical. 

\begin{lemma}\label{le:49}
Let $f$ be a homogeneous polynomial of degree~$\delta$.
\begin{enumerate}[(i)]
\item For a uniformly distributed~$w$ in the Euclidean unit ball of~$\mathbb{C}^{n+1}$ we have
  \[
    \|f\|_W^2 = \binom{n+1+\delta}{\delta} \mathbb{E} \left[ \abs{f(w)}^2 \right].
  \]
\item   For a uniformly distributed~$z$ in the unit sphere of~$\mathbb C^{n+1}$ we have 
  \[
    \|f\|_W^2 = \binom{n+\delta}{\delta} \mathbb{E} \left[ \abs{f(z)}^2 
\right].
  \]
\end{enumerate}
\end{lemma}

\begin{proof}
  Let~$H$ be the space of homogeneous polynomials of degree~$\delta$
  in~$z_0,\dotsc,z_n$. Both left-hand and right-hand sides of the first stated
  equality define a norm on~$H$ coming from a Hermitian inner product. The
  monomial basis is orthogonal for both{: this} is obvious for Weyl's norm. For
  the $L^2$-norm, this is \parencite[Proposition~1.4.8]{Rudin_1980}. So it only
  remains to check that the claim holds true when~$f$ is a monomial. By
  \parencite[Proposition~1.4.9(2)]{Rudin_1980},
  if~$w^{\alpha} = w_0^{\alpha_0}\cdots w_n^{\alpha_n}$ is a monomial of
  degree~$\delta$, {we have}
  \begin{align}
    \mathbb{E}\left[\abs{w^\alpha}^2 \right] &= \frac{(n+1)! \alpha_0!\dotsb \alpha_n!}{(n+1+\delta){!}} = \frac{(n+1)! \delta!}{(n+1+\delta)!} \cdot \frac{\alpha_0! \dotsb \alpha_n!}{\delta!},\\
    &= \tbinom{n+1+\delta}{\delta}^{-1} \left\| w^\alpha \right\|^2_W.
  \end{align}
  which is the claim.
  The second equality follows similarly from  \parencite[Proposition~1.4.9(1)]{Rudin_1980}.
\end{proof}

The following inequalities will also be useful.
\begin{lemma}\label{lem:weyl-norm-inequalities}
  For any homogeneous polynomial~$f \in \mathbb{C}[z_0,\dotsc,z_n]$ of degree~$\delta$,
  \begin{equation*}\label{eq:useful}
    \binom{n+\delta}{\delta}^{-1} \|f\|^2_W \leq 
    \max_{z\in \bS(\mathbb{C}^{n+1})} \abs{f(z)}^2  = \max_{w\in B(\mathbb{C}^{n+1})} \abs{f(w)}^2
    \leq \|f\|_W^2 . 
  \end{equation*}
\end{lemma}

\begin{proof}
  The first inequality follows directly from {the second equality of Lemma~\ref{le:49}.
  It is clear that the  maximum is reached on the boundary.}
  For the second inequality, we may assume (because of the unitary invariance of~$\|-\|_W$) that the maximum of~$|f|$ 
  on {the unit ball} is reached at~$(1,0,\dotsc,0)$.
  Besides, the coefficient~$c_{\delta,0,\dotsc,0}$ of~$f$ is~$f(1,0,\dotsc,0)$.
  Therefore,
  \[ \max_{w\in B} |f(w)|^2 = \abs{f(1,0,\dotsc,0)}^2 = \abs{c_{\delta,0,\dotsc,0}}^2 \leq \|f\|_W^2. \qedhere \]
\end{proof}

\subsection{Probabilistic evaluation of the gamma number} 
\label{sec:prob-eval-gamma}

The main reason for introducing the Frobenius norm in the $\gamma$ number,
instead of the usual operator norm, is the equality (Lemma~I.30)
\begin{equation}\label{eq:5}
\frac{1}{k!} \left\| \ud_z^kf \right\|_\Frob = \|f(z+\bullet)_k\|_W,
\end{equation}
where~$\|f(z+\bullet)_k\|_W$ is the Weyl norm 
 of the homogeneous component of degree~$k$
of the shifted polynomial~$x\mapsto f(z+x)$. It follows that
\begin{equation}\label{eq:52}
  \gamma_\Frob(f, z) = \sup_{k\geq 2} \left( \left\| \ud_zf \right\|^{-1}
  \|f(z+\bullet)_k\|_W \right)^{\frac{1}{k-1}}.
\end{equation}
This equality opens {up} interesting ways for estimating~$\gamma_\Frob$, and
therefore~$\gamma$. We used it to compute~$\gamma_\Frob$ efficiently when~$f$ is
a dense polynomial given in the monomial basis, see~\S{I.4.3.3}. In {that}
context, {we would} compute the shift~$f(z+\bullet)$ in the same monomial 
basis
in quasilinear time as~$\min(n,\delta)\to\infty$. From there, the
quantities~$ \|f(z+\bullet)_k\|_W$ can be computed in linear time. In the
black-box model, {however}, the monomial expansions (of either $f$ or
$f(z+\bullet)$) cannot fit into a~$\mop{poly}(n, \delta) L(f)$ complexity 
bound,
because the number of monomials of degree~$\delta$ in~$n+1$ variables is not
$\poly(n,\delta)$. Nonetheless, we can obtain a good enough approximation
of~$\|f(z+\bullet)_k\|_W$ with a few evaluations but a nonzero probability of
failure. This is the purpose of Algorithm~\ref{algo:gammaprob}, which we analyze
in the next theorem.

\begin{theorem}\label{thm:probabilistic-gamma}
  Given~$f \in \mathbb{C}[x_0,\dotsc,x_n]$ as a black-box function, an upper
  bound~$\delta$ on its degree, a point $z \in \mathbb{C}^{n+1}$, and
  some~$\epsilon > 0$, algorithm {\sc GammaProb} computes some~$\Gamma \geq 0$
  such that
  \[ \gamma_\Frob(f, z) \leq \Gamma \leq 192n^2\delta \cdot \gamma_\Frob(f, z) \]
  with probability at least~$1-\epsilon$,
  using~$O\left( \delta \log\left(\frac\delta\epsilon\right) (L(f) + n + \log\delta)\right)$
  operations.

  Moreover, for any~$t \ge 1$,
  \[
    \mathbb{P} \left[ \Gamma \leq \frac{\gamma_\Frob(f, z)}{t} \right]
    \leq \epsilon^{1+ \frac12 \log_2 t}.
  \]
\end{theorem}

\begin{algo}[tp]
  \centering
  \begin{algorithmic}
    \Function{\sc GammaProb}{$f$, $z$, $\epsilon$}
    \State $h \gets f(z + \bullet)$ (as black-box evaluation program)
    \State $s \gets \left\lceil 1 + \log_2 \frac{\delta}\epsilon \right\rceil$
    \For{$i$ from~1 to~$s$}
    \State $w_i \gets$ random uniformly distributed element of~$B$
    (unit ball of~$\mathbb{C}^{n+1}$)
    \State compute $h_2(w_i),\dotsc,h_{\deg f}(w_i)$, where~$h_k$ is the degree $k$
    component of~$h$
    \State \Comment Lemma~\ref{lem:complexity-homogeneous-components}
    \EndFor
    \State compute~$\ud_0 h$
    \State \Return $\displaystyle \max_{2\leq k \leq \delta}
    \left(\frac{({32} n k)^k}{\left\| \ud_0 h \right\|^{2}} \cdot {n+1+k\choose k}\frac1s \sum_{i=1}^s
    \left| h_k(w_i) \right|^2 \right)^{\frac{1}{2k-2}}$.
    \EndFunction
  \end{algorithmic}
  \caption[]{Probabilistic estimation of~$\gamma_\Frob$
  \begin{description}
  \item[Input:] $f \in \mathbb{C}[x_0,\dotsc,x_n]$ of degree~$\leq \delta$, given as black-box evaluation program,
    $z \in \mathbb{C}^{n+1}$, and $\epsilon > 0$
   \item[Output:] $\Gamma \in \mathbb{R}$
  \item[Postcondition:]
      $\gamma_\Frob(f, z) \leq \Gamma \leq {192 \, n^2\delta} \, \gamma_\Frob(f, z)$
        with probability at least~$1-\epsilon$.
    \end{description}
  }
  \label{algo:gammaprob}
\end{algo}

Note that we currently do not know how to estimate~$\gamma_\Frob$ within
an arbitrarily small factor.
The key in Theorem~\ref{thm:probabilistic-gamma} 
is to write each~$\|f(z+\bullet)_k\|_W^2$ as an expectation
(this is classical, see~\S\ref{se:weyl}) and to  approximate it by sampling (there are some obstacles).  We
assume that~$z = 0$ by changing~$f$ 
to $f(z + \bullet)$,
which is harmless because the
evaluation complexity is changed to~$L(f) + O(n)$.  Furthermore, the homogeneous
components~$f_k$ of~$f$ are accessible as black-box functions;
this is the content of the next lemma.

\begin{lemma}\label{lem:complexity-homogeneous-components}
Given~$w \in \bC^{n+1}$, one can compute~$f_0(w),\dotsc,f_\delta(w)$,
with~$O(\delta (L(f) + n + \log \delta))$ arithmetic operations.
\end{lemma}

\begin{proof}
We first compute all~$f(\xi^i w)$, for~$0\leq i \leq \delta$
for some primitive root of unity~$\xi$ of order~$\delta+1$.
This takes~$(\delta+1) L(f) + O(\delta n)$ arithmetic operations.
Since
\begin{equation}
  f(\xi^i w) = \sum_{k=0}^\delta \xi^{i k} f_k(w),
\end{equation}
we recover the numbers~$f_k(w)$ with the inverse Fourier transform,
\begin{equation}
  f_k(w) = \frac{1}{\delta+1}\sum_{i = 0}^\delta \xi^{-ik} f(\xi^i w).
\end{equation}
We may assume that $\delta$ is a power of two ($\delta$ is only required to 
be an upper bound on the degree of~$f$), and the fast Fourier transform
algorithm has an $O(\delta\log \delta)$ complexity bound to recover the~$f_k(z)$.
(With slightly more complicated formulas, we can also use~$\xi= 2$ to keep
close to the pure BSS model.)
\end{proof}

We now focus on the probabilistic estimation of~$\|f_k\|_W$ via a few
evaluations of~$f_k$. Let~$B \eqdef {B(\bC^{n+1})}$ denote the Euclidean unit
ball in~$\mathbb{C}^{n+1}$ and let~$w \in B$ be a uniformly distributed random
variable. By Lemma~\ref{le:49} we have
\begin{equation}\label{eq:49}
  \|f_k\|_W^2 = \binom{n+1+k}{k} \bE \left[ \abs{f_k(w)}^2 \right].
\end{equation}
The expectation in the right-hand side can be estimated with finitely
many samples of~$\abs{f_k(w)}^2$.
To obtain a rigorous confidence interval, we study some statistical
properties of~$\abs{f_k(w)}^2$.
Let~$w_1,\dotsc,w_s$ be independent uniformly distributed variables
in~$B$, and let
\begin{equation}
  \hat\mu_k^2 \eqdef \frac1s \sum_{i=1}^s \abs{f_k(w_i)}^2
\end{equation}
denote their \emph{empirical mean}.
Let~$\mu_k^2 \eqdef \EE[ |f_k(w)|^2 ] = \EE[ \hat\mu^2_k ]$ be the mean 
that
we want to estimate.
(Note that both $\mu_k$ and $ \hat\mu_k$ depend on $f_k$; 
we supressed this dependence in the notation.)

The next proposition shows that~$\hat\mu_k^2$ estimates~$\mu_k^2$ within a
$\poly(n,k)^{k}$ factor with very few samples. The upper bound is obtained by a
standard concentration inequality (Hoeffding's inequality). The lower bound is
more difficult, and very specific to the current setting, because we need 
to
bound~$\mu_k^2$ away from zero with only a small number of samples.
Concentration inequalities do not apply because the standard deviation may be
larger than the expectation, so a confidence interval whose radius is comparable
to the standard deviation (which is what we can hope for with a small number of
samples) may contain negative values.

\begin{proposition}\label{prop:bounds-empirical-mean}
For any~$0\leq k \leq \delta$, we have, with probability at least~$1 - 2^{1-s}$,
\[ ({32}nk)^{-k} \mu_k^2 \leq \hat\mu_k^2 \leq {(6n)^k} \mu_k^2, \]
where~$s$ is the number of samples.
\end{proposition}

Before proceeding with the proof, we state two lemmas, the principle of which
comes from \textcite[Lemma~8]{JiKollarShiffman_1992}.

\begin{lemma}\label{lem:distribution-small-values-univariate}
Let~$g \in \bC[z]$ be a univariate polynomial of degree~$k$ and let~$c \in \bC$ be
its leading coefficient. For any~$\eta > 0$,
\[
 \mop{vol} \left\{ z\in \mathbb{C} \st \abs{g(z)}^2 \leq \eta \right\}
 \leq \pi k \left({\abs{c}^{-2}}{\eta}\right)^{\frac1k}.
\]
\end{lemma}

\begin{proof}
Let~$u_1,\dotsc,u_k \in \mathbb{C}$ be the roots of~$g$, with
multiplicities, so that
\begin{equation}
  g(z) = c (z-u_1)\dotsb(z-u_k).
\end{equation}
The distance of some~$z\in \mathbb{C}$ to the set~$S \eqdef \left\{
  u_1,\dotsc,u_k \right\}$ 
is
the minimum of all~$\abs{z-u_i}$.
In particular
\begin{equation}
  \mop{dist}(z, S)^k \leq \prod_{i=1}^k \abs{z-u_i} = \abs{c}^{-1} \abs{g(z)}.
\end{equation}
Therefore,
\begin{equation}
  \left\{ z\in \mathbb{C} \st \abs{g(z)}^2 \leq \eta \right\}
  \subset \bigcup_{i=1}^k B \left( u_i, \abs{c}^{-\frac1k} \eta^{\frac{1}{2k}} \right),
\end{equation}
where~$B(u_i, r) \subseteq \mathbb{C}$ is the disk of radius~$r$ around~$u_i$.
The volume of~$B(u_i, r)$ is~$\pi r^2$, so the claim follows directly.
\end{proof}

\begin{lemma}\label{lem:distribution-small-values}
  If~$w \in B$ is a uniformly distributed random variable, then for all~$\eta > 0$,
\[
 \bP \left[ \abs{f_k(w)}^2 \leq \eta \max_\bS |f_k|^2 \right]
 \leq (n+1)k\eta^{\frac{1}{k}},
\]
\
where~$\max_\bS |f_k|$ is the maximum value of~$|f_k|$ on the unit sphere
in~$\mathbb{C}^{n+1}$. 
\end{lemma}

\begin{proof}
Let~$c$ be the coefficient of~$x_n^k$ in~$f_k$. It is the value of~$f_k$
at~$(0,\dotsc,0,1)$.
Up to a unitary change of coordinates, $|f_k|$ reaches a maximum at~$(0,\dotsc,0,1)$ so that~$c = \max_\bS |f_k|$.
Up to scaling, we may further assume that~$c=1$.
For any~$(p_0,\ldots,p_{n-1}) \in \bC^{n}$,
\begin{equation}
  \mop{vol} \left\{ z \in \bC \st \abs{f_k(p_0,\dotsc,p_{n-1}, z)}^2\leq \eta \right\}
  \leq \pi k {\eta}^{1/k},
\end{equation}
by Lemma~\ref{lem:distribution-small-values-univariate} applied to the
polynomial~$g(z) = f_k(p_0,\dotsc,p_{n-1}, z)$, which, by construction, 
is monic.
It follows, from the inclusion~$B(\mathbb{C}^{n+1})\subseteq B(\mathbb{C}^n) \times \mathbb{C}$, 
that
\begin{align}
  \MoveEqLeft \mop{vol} \left\{ w\in B(\mathbb{C}^{n+1}) \st \abs{f_k(w)}^2 \leq \eta \right\} \\
  &\leq \mop{vol}\left\{ (p_0,\dotsc,p_{n-1}, z)\in B(\mathbb{C}^{n}) \times \mathbb{C}
  \st \abs{f_k(p_0,\dotsc,p_{n-1}, z)}^2 \leq \eta \right\}\\
  &\leq \mop{vol} B(\mathbb{C}^{n}) \cdot \pi k {\eta}^{\frac1k}.
\end{align}
Using $\mop{vol} B(\mathbb{C}^{n}) = \frac{\pi^n}{n!}$
and dividing both sides by~$\mop{vol} B(\mathbb{C}^{n+1})$ 
concludes the proof.
\end{proof}

\begin{lemma}\label{lem:small-values-hatmuk}
For any~$\eta > 0$, we have
\[
  \bP \left[ \hat\mu_k^2 \leq \eta \mu^2_k \right] \leq \left(
  8 nk \eta^{\frac1k} \right)^{\frac s2}.
\]
\end{lemma}

\begin{proof}
Put $M \eqdef \max_\bS \abs{f_k}$.
If~$\hat\mu_k^2 \le \eta M^2$ then at least~$\lceil s/2 \rceil$ samples among $\abs{f(w_1)}^2,\dotsc,\abs{f(w_s)}^2$
satisfy~$\abs{f(w_i)}^2 \leq 2\eta M^2$.
By the union bound and Lemma~\ref{lem:distribution-small-values}
we obtain,
\begin{align}
  \bP \left[ \hat{\mu_k}^2 \leq \eta M^2 \right]
  &\leq \binom{s}{\lceil s/2 \rceil} \bP \left[ \abs{f(w)}^2
    \leq 2 \eta M^2 \right]^{\lceil s/2 \rceil} \\
  &\leq 2^s \left( (n+1)k \eta^{\frac1k} \right)^{\frac s2}\\
  &\leq \left( 8nk \eta^{\frac1k} \right)^{\frac{s}2}.
\end{align}
To conclude, we note that~$\mu_k \leq M$.
\end{proof}

\begin{proof}[Proof of Proposition~\ref{prop:bounds-empirical-mean}.]
With
  $\eta \eqdef \left(32 nk \right)^{-k}$,
Lemma~\ref{lem:small-values-hatmuk} gives
\begin{align}
  \bP \left[  \hat\mu_k^2 \le \eta \mu_k^2 \right]
  \leq \left( {8} n k \eta^{\frac1k} \right)^{\frac s2}
  = 2^{-s}.
\end{align}
It follows that
\begin{align}\label{eq:pprop3.5-1}
  \bP \left[ \mu_k^2 \leq   \left({32} n k\right)^{k} \hat \mu_k^2 \right] \geq 1-2^{-s},
\end{align}
which is the stated left-hand inequality.

For the right-hand inequality, we apply Hoeffding's inequality
\parencite[e.g.,][Theorem~2.8]{BoucheronLugosiMassart_2013}.
The variable~$s \hat \mu_k^2$ is a sum of~$s$ independent variables
lying in the interval~$[0, M^2]$,
where we again abbreviate $M \eqdef \max_\bS \abs{f_k}$.
Accordingly, 
for any~$C \geq 1$,
\begin{equation}
  \bP \left[ \hat \mu_k^2 \geq C \mu_k^2 \right]
  = \bP \left[ s \hat\mu_k^2 - s \mu^2_k \geq (C-1) s \mu^2_k \right] \leq \exp \left( - \frac{2 (C-1)^2 s^2 \mu_k^4}{sM^4} \right) .
\end{equation}
By Lemma~\ref{lem:weyl-norm-inequalities} combined with~\eqref{eq:49}, we 
have
\begin{equation}
  M^2 \le \binom{n+1+k}{k}  \mu^2_k.
\end{equation}
Applying this bound, we obtain
\begin{equation}
  \bP \left[ \hat \mu_k^2 \geq C \mu_k^2 \right] \leq
  \exp \left( - \frac{2 (C-1)^2 s}{\binom{n+1+k}{k}^2} \right).
\end{equation}
We choose~$C = (6n)^k$ and simplify further using the inequality
$\binom{m+k}{k} \le \frac{(m+k)^k}{k!} \le (e(m+k)/k)^k$ 
and $e(n+1+k)/k \le e(n+3)/2$ (use $k\ge 2$)
to obtain
\begin{align}
  \frac{C-1}{\binom{n+1+k}{k}} &\geq \frac{(6n)^k - 1 }{\left( \frac{e (n+3)}{2} \right)^{k}} \geq \left( \frac{12 n}{e (n+3)} \right)^{k} - \left( \frac{2}{e(n+3)} \right)^{k} \\
&\geq  \left( \frac{3}{e} \right)^2 - \left( \frac{2}{4 e} \right)^2 \geq 
\sqrt{\frac12 \log{2}}.
\end{align}
We obtain therefore
\begin{equation}
  \bP \left[ \hat \mu_k^2 \geq (6n)^k \mu_k^2 \right] \leq \exp (-\log(2) 
s) =  2^{-s} .\label{eq:61}
\end{equation}
Combined with \eqref{eq:pprop3.5-1},
the union bound implies
\begin{equation}
  \bP \left[ \hat \mu_k^2  \leq (32nk)^{-k} \mu_k^2 \mbox{ or } \hat \mu_k^2 \ge (6n)^k \mu_k^2 \right] \leq
  2 \cdot 2^{-s}
\end{equation}
and the proposition follows.
\end{proof}

\begin{proof}[Proof of Theorem~\ref{thm:probabilistic-gamma}.]
Recall that we assume that~$z=0$.
Proposition~\ref{prop:bounds-empirical-mean} can be rephrased as follows:
with probability at least $1-2^{1-s}$, we have 
\begin{equation}
\mu_k^2 \le (32nk)^{k} \hat\mu_k^2 \leq  (192n^2 k)^k \mu_k^2
\end{equation}
Defining
\begin{equation}
  c^2_k \eqdef \binom{n+1+k}{k} (32 n k)^k \hat\mu^2_k ,
\end{equation}
using that 
by \eqref{eq:49}
\begin{equation}\label{eq:49again}
  \|f_k\|_W^2 = \binom{n+1+k}{k} \mu_k^2 ,
\end{equation}
and applying the union bound,
we therefore see that 
\begin{equation}\label{eq:51}
  \|f_k\|^2_W \leq c^2_k \leq {(192 n^2 k)^k} \cdot \|f_k\|^2_W
\end{equation}
holds for all $2\le k\le \delta$, with probability at least $1-\delta 2^{1-s}$.
If we chose~$s = \left\lceil 1 + \log_2 \frac \delta\epsilon \right\rceil$,
then $\delta 2^{1-s} \leq \epsilon$.
Recall from~\eqref{eq:5} and~\eqref{eq:52} that
\begin{equation}\label{eq:3.31}
  \gamma_\Frob(f, z) =
   \max_{\delta \ge k\geq 2} \left( \left\| \ud_0 f \right\|^{-1}
  \|f_k \|_W \right)^{\frac{1}{k-1}}.
\end{equation}
Noting that~$(192 n^2 k)^{\frac{k}{k-1}} \leq (192 n^2 \delta)^2$, for~$2\leq k\leq \delta$,
we conclude that the random variable
\begin{equation}\label{eq:63}
  \Gamma \eqdef \max_{2 \leq k \leq \delta}
  \left( \|\ud_0 f\|^{-1} c_k \right)^{\frac{1}{k-1}},
\end{equation}
which is returned by Algorithm~\ref{algo:gammaprob},
indeed satisfies 
\begin{equation}
 \gamma_\Frob (f,z) \le \Gamma \le 192 n^2 \delta \cdot \gamma_\Frob (f,z)\label{eq:62}
\end{equation}
with probability at least~$1-\epsilon$,
which proves the first assertion.

For the assertion on the number of operations, it suffices to note that 
by Lemma~\ref{lem:complexity-homogeneous-components}, 
the computation of $d_0f$ and of~$\hat\mu_2,\dotsc,\hat\mu_\delta$ can be 
done with
$O(s\delta (L(f) + n + \log \delta))$ arithmetic operations.

It only remains to check, for any~$t \geq 1$, the tail bound
\begin{equation}\label{eq:64}
  \mathbb{P} \left[ \Gamma \leq \frac{\gamma_\Frob (f, z)}{t} \right] \leq \epsilon^{1+\frac12 \log_2 t}.
\end{equation}
Unfolding the definitions~\eqref{eq:3.31} and~\eqref{eq:63}
and using again \eqref{eq:49again},
we obtain
\begin{align}
  \mathbb{P} \left[ \Gamma \leq \frac{\gamma_\Frob (f, z)}{t} \right]
  &\leq \sum_{k=2}^\delta \mathbb{P} \left[ (32nk)^k \hat\mu_k^2 \leq t^{-2 (k-1)} \mu_k^2 \right] \\
  &\leq \sum_{k=2}^\delta  \left( 8 nk \cdot \left( (32nk)^k t^{2(k-1)} 
\right)^{-\frac1k} \right)^{\frac s2}, && \text{by Lemma~\ref{lem:small-values-hatmuk}}, \\
  &= \sum_{k=2}^\delta \left( \frac14 t^{-2\frac{k-1}{k}} \right)^{\frac s2} \leq \delta 2^{-s} t^{-\frac s2}.
\end{align}
Since $s = \left\lceil 1 + \log_2 \frac \delta\epsilon \right\rceil$, we have~$\delta 2^{-s} \leq \epsilon$.
Furthermore, $s \geq -\log_2 \epsilon$, so
\begin{equation}
  t^{-\frac s2} \leq t^{\frac 12 \log_2 \epsilon} = \epsilon^{\frac12 \log_2 t},
\end{equation}
which proves~\eqref{eq:64}.
\end{proof}

\subsection{A Monte-Carlo continuation algorithm}
\label{sec:cont-algor}

We deal here with the specifics of a numerical continuation with a step-length computation that may be wrong.

The randomized 
algorithm for the evaluation of the step length can be plugged
into the rigid continuation algorithm (Algorithm~\ref{algo:nc}). There is 
no
guarantee, however, that the randomized computations of the $\gamma_\Frob$ fall
within the confidence interval described in
Theorem~\ref{thm:probabilistic-gamma} and, consequently, 
there is no guarantee that the
corresponding step-length estimation is accurate. If step lengths are
underestimated, we don't control anymore the complexity: as the step lengths go to zero, the number of steps goes to infinity.
Overestimating a single step length, instead, may undermine
the correctness of the result, and the
subsequent behavior of the algorithm is unknown (it may even not to terminate).
So we introduce a limit on the number of continuation steps.
Algorithm~\ref{algo:nclim} is a corresponding modification of
Algorithm~\ref{algo:nc}. When reaching the limit on the number of steps, this
algorithm halts with a failure notification.

\begin{proposition}\label{prop:behaviour-nclim}
On input~$F$, $\bfu$, $\bfv$, $z$, $K_{\max}$, and $\epsilon$,
such that
$z$ is a zero of~$\bfv \cdot F$, 
the randomized
Algorithm~{\sc BoundedBlackBoxNC} either fails or returns some~$w \in \mathbb{P}^n$.
In the latter case, $w$ is an approximate zero of~$\bfu \cdot F$
with probability at least~$1-\epsilon$. 
The total number of operations
is $\poly(n,\delta) \cdot K_{\max} \log \left( K_{\max}
\epsilon^{-1} \right) \cdot L(F)$.
\end{proposition}

\begin{algo}[tp]
  \centering
  \begin{algorithmic}
    \Function{\sc BoundedBlackBoxNC}{$F$, $\bfu$, $\bfv$, $z$, $K_{\max}$, $\epsilon$}
    \State $\eta \gets (n K_{\max})^{-1} \epsilon$
    \State $(\bfw_t)_{0\leq t \leq T} \gets$ a 1-Lipschitz continuous path
    from~$\bfv$ to~$\bfu$ in~$\cU$
    \State{$t \gets 0$}
    \For{$k$ from 1 to $K_{\max}$}
    \For{$i$ from $1$ to $n$}
    \State $w \gets$ $i$th component of~$\bfw_t$
    \State $g_i \gets \textsc{GammaProb}(f_i\circ w^{-1}, z, \eta)$
    \Comment Algorithm~\ref{algo:gammaprob}
    \EndFor
    \State $t \gets t + \left(240 \, \kappa(\bfw_t, z)^2 \left(
        \sum_{i=1}^n g_i^2 \right)^{\frac12}\right)^{-1}$
    \If{$t \geq T$}
    \State \textbf{return} $z$
    \EndIf
    \State $z \gets \mop{Newton}(\bfw_t \cdot F, z)$
    \Comment Newton iteration
    \EndFor
    \State \textbf{return} \FAIL
    \EndFunction
  \end{algorithmic}
  \caption[]{Bounded-time numerical continuation routine for black-box input
    \begin{description}
    \item[Input:] $F \in\cH$ (given as black-box), $\bfu$, $\bfv \in \cU$, $z \in \Proj$,
      $K_{\max} > 0$ and~$\epsilon > 0$
      \item[Precondition:] $z$ is a zero of~$\bfv \cdot F$.
      \item[Output:] $w\in\bP^n$ or \FAIL.
      \item[Postcondition:] If some $w\in\bP^n$ is output then $w$ is an approximate zero
        of~$\bfu \cdot F$ with probability~$\geq 1-\epsilon$.  
    \end{description}
  }
\label{algo:nclim}
\end{algo}

\begin{proof}
Assume $w\in\bP^n$ is returned which is not an approximate zero of $F$. 
This {implies} that one of the
estimations of $\gamma_\Frob(f,z)$,
computed by the {\sc GammaProb} subroutines
yielded a result that is smaller than the actual value of~$\gamma_\Frob(f, z)$.
There are at most~$n K_{\max}$ such estimations, so
by Theorem~\ref{thm:probabilistic-gamma}, this 
happens with
probability at most $n K_{\max} \eta$, which by choice of~$\eta$ is exactly~$\epsilon$.

The total number of operations is bounded by~$K_{\max}$ times the
cost of an iteration. The cost of an iteration is dominated by the
evaluation of the~$g_i$, which is bounded
by $O(\delta \log(\delta n K_{\max} \epsilon^{-1}) (L(F) + n +\log\delta ) )$
by Theorem~\ref{thm:probabilistic-gamma} and the choice of $\eta$,
and the Newton iteration, which costs $\poly(n,\delta) L(F)$.
\end{proof}

In case Algorithm~\ref{algo:nclim} fails, it is natural to restart the
computation with a higher iteration limit.  This is
Algorithm~\ref{algo:nc-with-restart}.
We can compare its complexity to that of Algorithm~\ref{algo:nc}, which assumes an exact computation of~$\gamma$.
Let~$K(F, \bfu, \bfv, z)$ be a bound for the number of iterations performed
by Algorithm~\ref{algo:nc} on input $F$, $\bfu$, $\bfv$ and~$z$,
allowing an overestimation of the step length up to a factor~$192 n^2\delta$
(in view of Theorem~\ref{thm:probabilistic-gamma}).

\begin{algo}[tp]
  \centering
  \begin{algorithmic}
    \Function{\sc BlackBoxNC}{$F$, $\bfu$, $\bfv$, $z$, $\epsilon$}
      \State $K_{\max} \gets 1$
      \Repeat
      \State $K_{\max} \gets 2 K_{\max}$
      \State $w\gets \textsc{BoundedBlackBoxNC}\left(F, \bfu, \bfv, z, K_{\max}, \epsilon \right)$
      \Comment Algorithm~\ref{algo:nclim}
      \Until{$w \neq $ \FAIL}
      \State \Return $w$
    \EndFunction
  \end{algorithmic}
  \caption[]{Numerical continuation for black-box input
    \begin{description}
    \item[Input:] $F \in\cH$ (given as black-box), $\bfu$, $\bfv \in \cU$, $z \in \Proj$
      and~$\epsilon \in (0,\tfrac14 \rbrack$.
      \item[Precondition:] $z$ is a zero of~$\bfv \cdot F$.
      \item[Output:] $w\in\bP^n$  if algorithm terminates.
      \item[Postcondition:] $w$ is an approximate zero of~$\bfu \cdot F$
        with probability $\ge 1-\epsilon$.
    \end{description}
  }
  \label{algo:nc-with-restart}
\end{algo}

\begin{proposition}\label{prop:nc-with-restart}
On input~$F$, $\bfu$, $\bfv$, $z$ and $\epsilon \in (0,\frac14 \rbrack$,
such that
$z$ is a zero of~$\bfv \cdot F$, 
and $K(F, \bfu, \bfv, z) < \infty$, 
the randomized
Algorithm~\ref{algo:nc-with-restart} terminates almost surely and
returns an approximate zero of~$\bfu \cdot F$ with probability at
least~$1-\epsilon$. The average total number of operations
is $\poly(n, \delta) \cdot L(F) \cdot K \log\left( K \epsilon^{-1} \right)$,
with~$K = K(F, \bfu, \bfv, z)$. (NB: The only source of randomness is the probabilistic evaluation of~$\gamma_\Frob$.)
\end{proposition}

\begin{proof}
Let~$K\eqdef K(F, \bfu, \bfv, z)$. 
By definition of~$K$, if all approximations lie in the desired confidence
interval, then \textsc{BoundedBlackBoxNC} terminates after at most~$K$ iterations.
So as soon as~$\Kmax \geq K$, {\sc BoundedBlackBoxNC} may  return \textsc{Fail}
only if the approximation of some~$\gamma_\Frob$ is not correct.
This happens with probability at most~$\epsilon$ at each iteration of
the main loop in Algorithm~\ref{algo:nc-with-restart}, independently.
So the number of iterations is finite almost surely.
That the result is correct with probability at least~$1-\epsilon$ follows
from Proposition~\ref{prop:behaviour-nclim}.

We now consider the total cost.
At the~$m$th iteration, we have $\Kmax=2^m$, so the cost of the~$m$th iteration
is $\poly(n, \delta) \cdot 2^m \log(2^m \epsilon^{-1}) \cdot L(f)$, by
Proposition~\ref{prop:behaviour-nclim}.
Put $\ell \eqdef \lceil \log_2 K \rceil$.
If the $m$th iteration is reached for some~$m> \ell$,
then all the iterations from~$\ell$ to~$m-1$ have failed.
This has a probability~$\leq \epsilon^{m - \ell}$ to happen, so,
{if $I$ denotes the number of iterations}, we have
\begin{equation}
  \mathbb{P}[I \geq m] \leq \min\left(1, \epsilon^{m - \ell}\right).
  \label{eq:32}
\end{equation}
The total expected cost is therefore bounded by
\begin{align}
  \label{eq:33}
  \mathbb{E}[\mathrm{cost}]
  &\leq \poly(n, \delta)  L(f)\sum_{m=1}^\infty 2^m \log(2^m \epsilon^{-1})
    \mathbb{P}[I \geq m] \\
  &\leq \poly(n, \delta)  L(f)\sum_{m=1}^\infty 2^m \log (2^m \epsilon^{-1})
    \min\left(1, \epsilon^{m-\ell}\right).
\end{align}
The claim follows easily from splitting the sum into two parts,
$1\leq m < \ell$ and~$m > \ell$,
and applying the bounds (with $c=\log \epsilon^{-1}$)
\begin{equation}
  \sum_{m=1}^{\ell-1} 2^m (m + c) \leq (\ell+c)2^\ell
\end{equation}
and, for $\epsilon \in (0,\frac14)$,
\begin{equation}
  \sum_{m=\ell}^\infty 2^m (m+c) \epsilon^{m-\ell}
  \leq \frac{(\ell+c) 2^\ell}{(1-2\epsilon)^2} \leq 4 (\ell+c) 2^\ell. \hfill\qed\qedhere
\end{equation}
\end{proof}


\section{Condition based complexity analysis}
\label{sec:complexity-analysis}

We recall from Part~I (\S2) the {\em rigid solution variety} corresponding to a
polynomial system~$F = (f_1,\dotsc,f_n)$,
which consists of the pairs $(\bfv,z) \in \cU\times\bP^n$ such that
$(\bfv\cdot F)(z)=0$, which means
$f_1(v_1^{-1}z) =0,\ldots,f_n(v_n^{-1}z)=0$. 
To solve a given polynomial system~$F \in \cH$, we sample an initial pair
in the rigid solution variety corresponding to $F$ (Algorithm~I.1)
and perform a numerical continuation using Algorithm~\ref{algo:nc-with-restart}. 
This gives Algorithm~\ref{algo:solve} 
(recall that $\mathbf{1}_{\cU}$ denotes the unit in the group~$\cU$).
Termination and correctness directly follow from Proposition~\ref{prop:nc-with-restart}.

\begin{proposition}[Termination and correctness]\label{prop:correctness-termination}
  Let~$F=(f_1,\ldots,f_n)$ be a homogeneous polynomial system with
  only regular zeros.
 On input~$F$, given as a black-box evaluation program,
and~$\epsilon > 0$, Algorithm~{\sc BlackBoxSolve} terminates almost surely
and returns a point $z\in\mathbb{P}^n$ which is an approximate zero of~$F$
with probability~$1-\epsilon$.\eproof
\end{proposition}

\begin{theorem}[Complexity]\label{prop:complexity-without-validation}
Let~$F=(f_1,\ldots,f_n)$ be a homogeneous square-free polynomial system 
with
degrees at most~$\delta$ in $n+1$ variables given by a black-box evaluation program.
Let~$\bfu \in \cU$ be a random uniformly distributed and 
let~$H = \bfu \cdot F$.
Then, on input~$H$, given as black-box evaluation program,
and~$\epsilon > 0$, 
Algorithm~{\sc BlackBoxSolve} terminates after 
\[
\mop{poly}(n, \delta) \cdot L(F) \cdot \Gamma(F)
\big(\log\Gamma(F)+\log \epsilon^{-1}\big)
\]
operations on average.
   ``On average'' refers to expectation with respect to both
the random draws made by the algorithm and the random variable $\bfu$, but~$F$ is fixed.
\end{theorem}

\begin{algo}[tp]
  \centering
  \begin{algorithmic}
    \Function{\sc BlackBoxSolve}{$F$, $\epsilon$}
    \State Sample $(\bfv, z)$ in the rigid solution variety of~$F$
    \Comment Algorithm~I.1
    \State \Return $\textsc{BlackBoxNC} \left(F, \mathbf{1}_{\cU}, \bfv, z, \epsilon \right)$
    \Comment Algorithm~\ref{algo:nc-with-restart}
    \EndFunction
  \end{algorithmic}
  \caption[]{Zero finding for black-box input
    \begin{description}
    \item[Input:] $F \in\cH$  (given as black-box) 
      and~$\epsilon \in (0,\tfrac14 \rbrack$.
      \item[Output:] $w\in\bP^n$ if algorithm terminates.
      \item[Postcondition:] $w$ is an approximate zero of~$F$
        with probability $\geq 1-\epsilon$.
    \end{description}
  }
  \label{algo:solve}
\end{algo}

We next focus on
proving the complexity bound. Note that the statement of
Theorem~\ref{prop:complexity-without-validation}
is similar to that of Theorem~\ref{thm:first-main-result-complexity}. The 
only difference
lies in the complexity bound, whose dependence on $\epsilon^{-1}$ is logarithmic
in the former and doubly logarithmic in the latter.

\subsection{Complexity of sampling the rigid solution variety}

Toward the proof of Theorems~\ref{thm:first-main-result-complexity}
and~\ref{prop:complexity-without-validation}, we first review the
complexity of sampling the initial pair for the numerical continuation. In the
rigid setting, this sampling boils down to
sampling hypersurfaces, which in turn amounts to computing roots of 
univariate polynomials (see Part~I, \S2.4). Some technicalities are required to
connect known results about root-finding algorithms to our setting, and
especially the parameter~$\Gamma(F)$, but the material is very classical.

\begin{proposition}\label{prop:cost-sampling}
Given~$F \in \cH$ as a black-box evaluation program, we can sample
$\bfv \in \cU$ and~$\zeta \in \mathbb{P}^n$ such that~$\bfv$ is
uniformly distributed and
$\zeta$ is a uniformly distributed zero of~$\bfv \cdot F$, 
with~$\mop{poly}(n, \delta) \cdot ( L(F) + \log\log \Gamma(F) )$ 
operations on average.
\end{proposition}

\begin{proof}
This follows from Proposition~I.10 and Proposition~\ref{prop:sample-hypersurface}
below.
\end{proof}

\begin{proposition}\label{prop:sample-hypersurface}
For any~$f \in \mathbb{C}[z_0,\dotsc,z_n]$ homogeneous of degree~$\delta\ge 2$,
given as a black-box evaluation program,
one can sample a uniformly distributed point in the zero set~$V(f)$
of $f$ by a probabilistic algorithm
with $\mop{poly}(n, \delta) \cdot ( L(f) + \log\log \Gamma(f) )$
operations on average. 
\end{proposition}

\begin{proof}
Following Corollary~I.9, we can compute a uniformly distributed zero of~$f$ by
first sampling 
a line~$\ell \subset \mathbb{P}^n$ uniformly distributed in the Grasmannian of lines,
and then
sampling a uniformly distributed point in the finite set~$\ell\cap V(f)$.
To do this, we consider the restriction~$f|_\ell$, which, after choosing a
orthonormal basis of~$\ell$, is a bivariate homogeneous polynomial, and compute
its roots. The representation of~$f|_\ell$ in a monomial basis can be computed by
$\delta+1$ evaluations of~$f$
and interpolation, 
at a cost~$O \left( \delta (L(f) + n + \log\delta) \right)$,
{as in Lemma~\ref{lem:complexity-homogeneous-components}}.
By Lemma~\ref{lem:unisolve} below, computing the roots takes
\begin{equation}
    \poly(\delta) \log\log \left(\max_{\zeta \in \ell\cap V(f)} \gamma(f|_\ell, \zeta)\right)
    \label{eq:45}
\end{equation}
operations on average. We assume a 6th type of node to refine approximate 
roots into
exact roots 
{(recall the discussion in \S\ref{sec:BB}).}
Then we have,
by the definition~\eqref{eq:defGamma} of $\Gamma(f|_\ell)$,
\begin{equation}
  \max_{\zeta \in \ell\cap V(f)} \gamma(f|_\ell, \zeta)^{ 2} \leq
  \ \sum_{\mathclap{\zeta \in \ell\cap V(f)}} \gamma_{  \Frob}(f|_\ell, \zeta)^{ 2} = \delta \, \Gamma(f|_\ell)^{ 2}.
\end{equation}
Note that
$\delta \Gamma(f|_\ell)^2 \ge \delta \frac14 (\delta-1)^2 \ge \frac12 \ge 
\frac{1}{e}$
since $\gamma(f|_\ell,\zeta) \ge \frac12 (\delta-1)$ by Lemma~11 of Part~I. 
By Jensen's inequality, 
using the concavity of~$\log\log$ {on $\lbrack e^{-1},\infty)$}, 
we obtain 
\begin{equation}
 \EE_{\ell} \left[ \log\log\left( \delta\, \Gamma(f|_\ell)^2 \right) \right]
  \leq \log\log \left(\delta \EE_{\ell} \left[\Gamma(f|_\ell)^2  \right]\right) .
\end{equation}
Finally,  Lemma~\ref{lem:average-gamma-of-restriction} below gives 
\begin{equation}
  \log\log \left(\delta \EE_{\ell} \left[\Gamma(f|_\ell)^2  \right]\right) 
  \leq \log \log \left( 2n\delta \, \Gamma(f)^2 \right) 
\end{equation}
and the claim follows.
\end{proof}

\begin{lemma}\label{lem:unisolve}
Let~$g \in \mathbb{C}[{z_0,z_1}]$ be a homogeneous polynomial of
degree~$\delta$ without multiple zeros.  One can compute, with a probabilistic
algorithm, $\delta$ approximate zeros of~$g$, one for each zero of~$g$,
with $\poly(\delta) \log\log \gamma_{\max}$ operations on average,
where~$\gamma_{\max} \eqdef\max_{\zeta\in V(g)} \gamma(g, \zeta)$.
\end{lemma}

\begin{proof}
The proof essentially relies on the following known fact due to 
\textcite{Renegar_1987} (see
also~\cite[Thm.~2.1.1 and Cor.~2.1.2]{Pan_2001}, for tighter bounds).
Let~$f \in \mathbb{C}[t]$ be a given polynomial of degree~$\delta$, 
$R> 0$ be a known upper bound on the modulus of the roots~$\xi_1,\dotsc,\xi_\delta \in \mathbb{C}$ of~$f$, 
and $\epsilon>0$ be given.  
We can compute from this data with~$\poly(\delta) \log\log \frac R\epsilon$ operations
approximations $x_1,\dotsc,x_n \in \mathbb{C}$ 
of the zeros such that~$\abs{\xi_i - x_i} \leq \epsilon$.  

To apply this result to the given homogeneous polynomial~$g$,
we first apply a uniformly random unitary transformation $u\in U(2)$ to the given $g$
and dehomogenize $u\cdot g$,
obtaining the univariate polynomial $f\in\bC[t]$.

We first claim that with probability at least $3/4$ we have: ($*$)
$|\xi_i| \le 2 \sqrt{\delta}$ for all zeros $\xi_i\in\bC$ of $f$. This can be
seen as follows. We measure distances in $\bP^1$ with respect to the
projective (angular) distance. The disk of radius $\theta$ around a point 
in
$\bP^1$, has measure at most $\pi (\sin\theta)^2$
\parencite[Lemma~20.8]{BurgisserCucker_2013}. Let
$\sin\theta=(2\sqrt{\delta})^{-1}$. Then a uniformly random point~$p$ in
$\bP^1$ lies in a disk of radius~$\theta$ around a root of~$f$ with probability
at most $\delta (\sin\theta)^2 \le 1/4$.
Write $0\eqdef [1:0]$ and $\infty \eqdef [0:1]$ and note that 
$\mop{dist}(0,p) + \mop{dist}(p, \infty) =\pi/2$ for any $p\in\bP^1$.  
Since~$u^{-1}(\infty)$ is uniformly distributed, we conclude that with probability
at least~$3/4$, each zero $\zeta\in\bP^1$ of $g$ satisfies 
$\mop{dist}(\zeta, u^{-1}(\infty)) \ge \theta$, which means
$\mop{dist}(\zeta, u^{-1}( 0)) \le \pi/2 -\theta$.
The latter easily implies for the corresponding affine root
$\xi=\zeta_1/\zeta_0$ of~$f$ that
$|\xi| \le (\tan\theta)^{-1} \le (\sin\theta)^{-1}= 2\sqrt{\delta}$, 
hence ($*$) holds.

The maximum norm of a zero of $f\in\bC[{t}]$ can be computed with a small
relative error with $O(\delta \log \delta)$ operations
\parencite[Fact~2.2(b)]{Pan_1996}, so we can test the property~($*$). We
repeatedly sample a new~$u \in U(2)$ until ($*$) holds. Each iteration succeeds
with probability at least~$\frac34$ of success, so there are at most two
iterations on average.

For a chosen $\epsilon>0$, 
we can now compute with Renegar's algorithm the roots of~$f$, 
up to precision~$\epsilon$ 
with $\poly(\delta) \log\log \frac{1}{\epsilon}$ operations (where the
$\log\log 2\sqrt{\delta}$ is absorbed by~$\poly(\delta)$).
By homogeneizing and transforming back with $u^{-1}$, 
we obtain approximations~$p_1,\dotsc,p_\delta$ of the projective
roots $\zeta_1,\ldots,\zeta_\delta$ 
of~$g$ up to precision~$\epsilon$, measured in projective distance.

The remaining difficulty is that the $p_i$ might not be 
approximate roots of $g$, in the sense of Smale.
However, suppose that for all $i$ we have 
\begin{equation}\label{eq:65}
  \epsilon \gamma(g,p_i) \leq \tfrac{1}{11}.
\end{equation}
Using that $z \mapsto \gamma(g,z)^{-1}$ is 5-Lipschitz continuous on~$\mathbb{P}^1$ 
\parencite[Lemma 31]{Lairez_2020}, we see that 
$\epsilon \gamma(g, \zeta_i) \leq \frac16$ 
for all~$i$. This is known to imply that $p_i$ is an approximate zero of $p_i$ 
(\cite{ShubSmale_1993b}, and Theorem~I.12 for the constant).
On the other hand, using again the Lipschitz property, 
we are sure that Condition~\eqref{eq:65} is met 
as soon as $\epsilon \gamma_{\max} \leq \tfrac{1}{16}$.

So starting with~$\epsilon = \frac12$, we compute points~$p_1,\dotsc,p_\delta$
approximating~$\zeta_1,\dotsc,\zeta_\delta$ up to precision~$\epsilon$
until~\eqref{eq:65} is met for all~$p_i$, squaring~$\epsilon$ after each
unsuccessful iteration.
Note that Renegar's algorithm need not be restarted when~$\epsilon$ is refined.
We have $\epsilon \gamma_{\max} \leq \tfrac{1}{16}$ after 
at most~$\log\log( 16\gamma_{\max} )$ iterations. 
Finally, note that we do not need to compute exactly~$\gamma$, an
approximation within factor~2 is enough, with appropriate modifications of the constants, 
and this is achieved by~$\gamma_\Frob$,
see~\eqref{eq:67}, which we can compute in~$\poly(\delta)$ operations.
\end{proof}



\begin{lemma}\label{lem:average-gamma-of-restriction}
Let~$f \in \bC[z_0,\ldots,z_n]$ be homogeneous of degree $\delta$
and let~$\ell \subset \mathbb{P}^n$ be a uniformly distributed random projective line.
Then
  $\EE_{\ell} \left[ \Gamma(f|_\ell)^2 \right] \leq 2n\; \Gamma(f)^2$.
\end{lemma}

\begin{proof}
  Let~$\ell \subset \mathbb{P}^n$ be a uniformly distributed random projective line
  and let~$\zeta \in \ell$ be uniformly distributed among the zeros of~$f|_\ell$.
  Then $\zeta$ is also a uniformly distributed zero of~$f$, see Corollary~I.9. 
  Let~$\theta$ denote the angle between the tangent line $T_\zeta \ell$ and 
  the line $T_\zeta V(f)^{\perp}$ normal to $V(f)$ at $\zeta$.
 By an elementary geometric reasoning, we have 
 $\|\ud_\zeta f|_\ell\| = \|\ud_\zeta f\| \cos \vartheta(\ell,\zeta)$.
 Moreover, $\| \ud_\zeta^k f|_\ell \|_\Frob \leq \| \ud^k_\zeta f\|_\Frob$.
 So it follows that
  \begin{equation}\label{eq:74}
    \gamma_\Frob(f|_\ell,\zeta)^2 \leq \gamma_\Frob(f, \zeta)^2\cos(\theta)^{-2} .
  \end{equation}

  In order to bound this, we consider now a related, but different distribution.
  As above, let $\zeta$ be a uniformly distributed zero of~$f$. 
  Consider now a uniformly distributed random projective line $\ell'$ passing through~$\zeta$.
  The two distributions $(\ell,\zeta)$ and $(\zeta,\ell')$ are related by 
Lemma~I.5 as follows: 
  for any integrable function~$h$ of~$\ell$ and~$\zeta$, we have 
  \begin{equation}
    \mathbb{E}_{\ell,\zeta} [ h(\ell, \zeta) ] = c\, \mathbb{E}_{\zeta,\ell'} [ h(\ell', \zeta) \det^\perp ({T}_\zeta \ell', {T}_\zeta V(f)) ],
  \end{equation}
  where~$c$ is some normalization constant and where~$\det^\perp ({T}_\zeta \ell', {T}_\zeta V(f))$ is defined in~I.\S.2.1.
  It is only a matter of unfolding definitions to see that it is equal to~$\cos \theta'$, where 
  $\theta'$ denotes the angle between~$T_\zeta \ell'$ and~$T_\zeta V(f)^\perp$.
  With~$h = 1$, we obtain $c = \mathbb{E} \left[ \cos \theta' \right]^{-1}$ and therefore we get
  \begin{equation}\label{eq:50}
    \mathbb{E}_{\ell, \zeta} [ h(\ell, \zeta) ] = \mathbb{E}_{\zeta, \ell'} [ h(\ell', \zeta) \cos\theta' ] \, \mathbb{E}[\cos \theta']^{-1}.
  \end{equation}

  We analyze now the distribution of $\theta'$: 
  $\cos(\theta')^2$ is a beta-distributed variable with
  parameters~$1$ and~$n-1$: indeed, $\cos(\theta')^2 = \abs{u_1}^2 / \|u\|^2$
  where~$u\in \mathbb{C}^n$ is a Gaussian random vector, and it is well
  known that the distribution of this quotient of $\chi^2$-distributed
  random variables is a beta-distributed variable.
  Generally, the moments of a beta-distributed random variable $Z$ with parameters $\alpha,\beta$ satisfy 
 \begin{equation}\label{eq:mom-beta}
    \mathbb{E}[ Z^{r} ] = \frac{B(\alpha + r, \beta)}{B(\alpha,\beta)} ,
  \end{equation}
 where~$B$ is the Beta function and $r>-\alpha$. In particular, for~$r > -1$,
  \begin{equation}\label{eq:75}
    \mathbb{E}_{\zeta, \ell'}[ \cos(\theta')^{2r} ] = \frac{B(1+r, n-1)}{B(1,n-1)} ,
  \end{equation}
  and hence 
  \begin{equation}\label{eq:BBQ}
    \mathbb{E} \left[ \cos(\theta')^{-1} \right] \mathbb{E} \left[ \cos(\theta') \right]^{-1} = \frac{B(\frac12, n-1)}{B(\frac32, n-1)} = 2n - 
1.
  \end{equation}
  Continuing with~\eqref{eq:50}, we obtain
  \begin{align}
    \mathbb{E}_\ell \left[ \Gamma(f|_\ell)^2 \right] 
            &= \mathbb{E}_{\ell, \zeta} \left[ \gamma_\Frob(f|_\ell, \zeta)^2 \right], && \text{by~\eqref{eq:defGamma},} \\
            &\leq  \mathbb{E}_{\ell, \zeta} \left[ \gamma_\Frob(f, \zeta)^2 \cos(\theta)^{-2} \right], &&\text{by~\eqref{eq:74},} \\
            &= \mathbb{E}_{\zeta, \ell'} \left[ \gamma_\Frob(f, \zeta)^2 \cos(\theta')^{-1} \right] \mathbb{E} \left[ \cos(\theta') \right]^{-1}, &&\text{by~\eqref{eq:50},} \\
            &= \mathbb{E}_{\zeta} \left[ \gamma_\Frob(f, \zeta)^2 \right] \mathbb{E}_{\ell'} \left[\cos(\theta')^{-1} \right] \mathbb{E} \left[ \cos(\theta') \right]^{-1} \label{here} \\
            &= \Gamma(f)^2 (2n-1), && \text{by \eqref{eq:BBQ}}\,
  \end{align}
  the second last equality \eqref{here} since the random variable~$\theta'$ is independent from~$\zeta$.
  This concludes the proof.
\end{proof}

\subsection{Proof of Theorem~\ref{prop:complexity-without-validation}}
\label{sec:proof-of-main-result-1}

We now study the average complexity of
the algorithm $\textsc{BlackBoxSolve}(\bfu \cdot F, \epsilon)$, where~$\bfu \in
\cU$ is uniformly distributed.
Recall that~$\Gamma(\bfu\cdot F)=
\Gamma(F)$, by unitary invariance of~$\gamma_\Frob$, 
and~$L(\bfu \cdot F) = L(F) + O(n^3)$.

The sampling operation costs {at most}~$\mop{poly}(n, \delta) \cdot L(F) \cdot
\log \log \Gamma(F)$ on average, by
Proposition~\ref{prop:cost-sampling}.  The expected cost of the continuation
phase is~$\mop{poly}(n, \delta) \cdot L(F) \cdot K (\log K +\log\epsilon^{-1})$,
by Proposition~\ref{prop:nc-with-restart},
where~$K=K(\bfu \cdot F, \mathbf{1}_{\cU}, \bfv, z)$ and~$(\bfv, z)$ is 
the
sampled initial pair. By unitary invariance,
\begin{equation}
  K(\bfu\cdot F, \mathbf{1}_{\cU}, \bfv, z) = K(F, \bfu, \bfv', z),
\end{equation}
where~$\bfv' = \bfv\bfu$. Moreover, since~$\bfv$ is uniformly distributed
and independent from~$\bfu$, $\bfv'$ is also uniformly distributed and
independent from~$\bfu$, and~$z$ is a uniformly distributed zero
of~$\bfv' \cdot F$.
So the following proposition concludes the proof of
Theorem~\ref{prop:complexity-without-validation}.

\begin{proposition}
Let~$\bfu, \bfv \in \cU$ be independent and uniformly distributed random variables,
let~$\zeta$ be a uniformly distributed zero of~$\bfv\cdot F$ and
let~$K = K(F, \bfu, \bfv, \zeta)$. Then we have 
{$\EE \left[ K \right] \leq \poly(n, \delta) \Gamma(F)$} and
$\EE \left[ K \log K \right] \leq \poly(n, \delta) \cdot \Gamma(F) \log \Gamma(F)$.
\end{proposition}

\begin{proof}[Sketch of proof]
{The first bound ~$\EE \left[ K \right] \leq \poly(n, \delta) \Gamma(F)$
was shown in Theorem~I.27.}   
Following \emph{mutatis mutandis} the proof of Theorem~I.25 (the only additional fact
needed is Proposition~\ref{prop:moments-kappa} below {for $a=3/2$}),
we obtain that
\begin{equation}
    \EE \left[ K^{\frac32} \right] \leq  \poly(n, \delta) \Gamma(F)^{\frac32}.
\end{equation}
 
Next, we observe that the function~$h : x \mapsto x^{\frac23} (1+ \log x^{\frac23})$
is concave on~$[1,\infty)$.
By Jensen's inequalities, it follows that
\begin{equation}
  \EE \left[ K \log K \right] \leq
  \EE \left[ h( K^{\frac32} ) \right]
  \leq h \left( \EE [K^{\frac32} ]\right)
  \leq \poly(n, \delta) \Gamma(F) \log \Gamma(F),
\end{equation}
which gives the claim.
\end{proof}

The following statement extends Proposition~I.17 to more general 
exponents. The proof technique is more elementary and the result,
although not as tight, good enough for our purpose.

\begin{proposition}\label{prop:moments-kappa}
Let~$M \in \mathbb{C}^{n\times (n+1)}$ be a random matrix whose rows
are independent uniformly distributed vectors
in~$\mathbb{S}(\mathbb{C}^{n+1})$, and let~$\sigma_{\min}(M)$ be
the smallest singular value of~$M$. 
For all~$a \in [1,2)$,
\[ \EE \left[ \sigma_{\min}(M)^{-2a} \right] \leq \frac{n^{1+2a}}{2-a}, \]
and, equivalently with the notations of Proposition~I.17,
\[
  \EE \left[ \kappa(\bfu, \zeta)^{2a} \right] \leq \frac{n^{1+2a}}{2-a}.
\]
\end{proposition}

\begin{proof} 
For short, let~$\sigma$ denote~$\sigma_{\min}(M)$.
Let~$u_1,\dotsc,u_n$ be the rows of~$M$.
By definition, there is a unit vector~$x \in \mathbb{C}^n$ such that
\begin{equation}\label{eq:46}
  \| x_1 u_1 + \cdots + x_n u_n \|^2 = \sigma^2.
\end{equation}
If~$V_i$ denotes the subspace of~$\mathbb{C}^{n+1}$ {spanned} by all~$u_j$ except~$u_i$,
and $b_i$ denotes the squared Euclidean distance of $u_i$ to $V_i$, 
then~\eqref{eq:46} implies $b_i \leq \abs{x_i}^{-2} \sigma^2$
for all~$i$.
Moreover, since {$x$ is a unit vector},
there is at least one~$i$ such that~$\abs{x_i}^2 \geq\frac{1}{n}$.  
Hence $n \sigma^2 \geq \min_i b_i$ and therefore
\begin{equation}\label{eq:47}
  \EE \left[ \sigma^{-2a} \right]
  \leq {n^a \mathbb{E} \left[\max_i  b_i^{-a} \right]}
  \leq n^a \sum_{i=1}^n \mathbb{E}\left[b_i^{-a}\right].
\end{equation}

{To analyze the distribution of $b_i$ consider, for fixed $V_i$,}
a standard Gaussian vector~$p_i$ in~$V_i$, 
and an  independent standard Gaussian vector~$q_i$ in~$V_i^\perp$.
{(Note $\dim V_i^\perp =4$.)}
Since~$u_i$ is uniformly distributed in the sphere, it has the same distribution as
{$(p_i+q_i)/\sqrt{\|p_i\|^2 + \|q_i\|^2}$}.
In particular, {$b_i$~has the same distribution as
$\|q_i\|^2/(\|p_i\|^2 + \|q_i\|^2)$,}
which is a Beta distribution with parameters $2, n-1$,
since~$\|p_i\|^2$ and~$\|q_i\|^2$ are independent~$\chi^2$-distributed random
variables with~$2n-2$ and~$4$ degrees of freedom, respectively.
By \eqref{eq:mom-beta} we have for the moments, using $a<2$,
\begin{equation}
\EE \left[ b_i^{-a} \right] = \frac{B(2-a,n-1) }{B(2,n-1)}  
 = \frac{\Gamma(2-a) \Gamma(n+1)}{\Gamma(n+1-a)} .\label{eq:78}
\end{equation}
We obtain
\begin{align}
  \EE \left[ b_i^{-a} \right]
  & = \frac{\Gamma(3-a)}{2-a} \cdot n \cdot 
  \frac{\Gamma(n)}{\Gamma(n+1-a)},
  && \text{using twice~$\Gamma(x+1)=x\Gamma(x)$,}\\
  &\leq \frac{1}{2-a} \cdot n \cdot n^{a-1},
  && \text{by Gautschi's inequality.}
\end{align}
In combination with \eqref{eq:47} this gives the result.
\end{proof}

\subsection{Confidence boosting and proof of Theorem~\ref{thm:first-main-result-complexity}}
\label{sec:valid-proof-theor}

We may leverage the quadratic convergence of Newton's iteration to increase the confidence in
the result of Algorithm~\ref{algo:solve} and reduce the dependence
on~$\epsilon$ (the maximum probability of failure) from~$\log
\frac1\epsilon$ down to~$\log\log\frac1\epsilon$, so that we can choose
$\epsilon = \smash{10^{-10^{100}}}$ without afterthoughts, at least in
the BSS model. On a physical computer, the working precision should be
comparable with~$\epsilon$, which imposes some limitations.  A
complete certification, without possibility of error, with
$\poly(n,\delta)$ evaluations of~$F$, seems difficult to reach in the
black-box model: with only~$\poly(n,\delta)$ evaluations, we cannot
distinguish a polynomial system~$F$ from the infinitely many other
systems with the same evaluations. 

To describe this boosting procedure we first recall some details
about $\alpha$-theory and Part~I. Let~$F\in\cH$ be a polynomial system
and~$z\in\mathbb{P}^n$ be a projective point.  Let~$\cN_F(z)$ denote
the projective Newton iteration (and~$\cN^k_F(z)$ denote the composition
of~$k$ projective Newton iterations).  Let
\begin{equation}
  \beta(F, z) \eqdef d_\bP \left( z, \cN_F(z) \right).
\end{equation}
There is an absolute constant~$\alpha_0$ such that
for any~$z\in\mathbb{P}^n$,
if~$\beta(F, z) \gamma(F, z) \leq \alpha_0$,
then~$z$ is an approximate zero of~$F$ \parencite[Theorem~1]{DedieuShub_1999}.
This is one of many variants of the alpha-theorem of \textcite{Smale_1986}.
There may be differences in the definition of~$\gamma$ or $\beta$, or even
the precise definition of approximate zero, but they only change the
constant~$\alpha_0$.

It is important to be slightly more precise about the output of
Algorithm~\ref{algo:solve} (when all estimates are correct, naturally):
by the design of the numerical continuation (see Proposition~I.22 with~$C=15$ and~$A=\frac{1}{4C}$), the
output point~$w\in \mathbb{P}^n$ satisfies 
\begin{equation}\label{eq:53}
  d_\bP(w,\zeta) \hat\gamma_\Frob(F, \zeta) \leq  \frac{1}{4\cdot 15} = 
\frac{1}{60},
\end{equation}
for some zero~$\zeta$ of~$F$, where $\hat\gamma_\Frob$ is the
split Frobenius $\gamma$ number (see \S\ref{sec:split-gamma-number}).
This implies (see Theorem~I.12), using~$\gamma \leq \hat\gamma_\Frob$, that
\begin{equation}\label{eq:54}
  d_\bP \left( \cN_F^k(w), \zeta \right) \leq 2^{1-2^k} d_\bP (w, \zeta).
\end{equation}
The last important property we recall is the 15-Lipschitz continuity of the
function~$z\in\mathbb{P}^n \mapsto \hat\gamma_\Frob(F, z)^{-1}$ (Lemmas~I.26 and~I.31).

\begin{algo}[tp]
  \centering
  \begin{algorithmic}
    \Function{\sc Boost}{$F$, $w$, $\epsilon$}
    \State $k \gets \left\lceil \max \left( 1+ \log_2 \log_2
    \left( 20 n^2 \delta \alpha_0^{-1} \right), 1 + \log_2 \log_2 \epsilon^{-1} \right) \right\rceil$
    \State $z\gets \cN_F^k(w)$
    \State $c \gets \kappa(F, z) \left( \sum_{i=1}^n \textsc{GammaProb}
    (f_i, z,  \frac{1}{4n})^2 \right)^{\frac12}$
    \Comment Algorithm~\ref{algo:gammaprob}
    \If{ $2^{2^{k-1}} \beta(F, z) c \leq \alpha_0$ }
    \State \Return $z$
    \Else
    \State \Return \FAIL
    \EndIf
    \EndFunction
  \end{algorithmic}
  \caption[]{Boosting the confidence for approximate zeros
    \begin{description}
    \item[Input:] $F = (f_1,\dotsc,f_n) \in\cH$ (given as black-box), $w \in \mathbb{P}^n$ and~$\epsilon\in (0,\frac12)$.
      \item[Output:] $z \in \mathbb{P}^n$ or \FAIL
      \item[Postcondition:] If {\sc Boost} returns a point~$z$, then it is an
        approximate zero of~$F$ with probability~$\geq 1-\epsilon$.
    \end{description}
  }
  \label{algo:validate}
\end{algo}

Algorithm~\ref{algo:validate} checks the criterion~$\beta(F, z)
\gamma(F, z) \leq \alpha_0$ after having refined the presumed
approximate zero with a few Newton's iterations.  If the input point is
indeed an approximate zero, then~$\beta(F, z)$ will be very small and
it will satisfy the criterion above even with a very gross approximation
of~$\gamma(F, z)$.

\begin{proposition}
On input~$F\in\cH$, $w\in\mathbb{P}^n$, and~$\epsilon \in (0,\frac12)$,
Algorithm~\textsc{Boost} outputs some~$z \in \mathbb{P}^n$ (succeeds)
or fails after~$\poly(n, \delta)L(F) \log \log \epsilon^{-1}$ operations.
If~$w$ satisfies~\eqref{eq:53}, then 
it succeeds with probability at least~$\frac34$.
If it succeeds, then the output point is an approximate zero of~$F$ with
probability at least~$1-\epsilon$.
\end{proposition}

\begin{proof}
We use the notations ($k$, $z$, and $c$) of Algorithm~\ref{algo:validate}.
Assume first that~\eqref{eq:53} holds for $w$ and some zero~$\zeta$ of~$F$.
By~\eqref{eq:54} and~\eqref{eq:53},
\begin{equation}\label{eq:55}
  d_\bP(z, \zeta) \hat\gamma_\Frob(F, \zeta)
  \leq \frac{1}{60}.
\end{equation}
Using the Lipschitz continuity and~\eqref{eq:55},
\begin{equation}
  \hat\gamma_\Frob(F, z) \leq \frac{\hat\gamma_\Frob(F, \zeta)}{1-15 d_\bP(z, \zeta)
    \hat\gamma_\Frob(F, \zeta)} \leq \frac43 \hat\gamma_\Frob(F, \zeta),
\end{equation}
and it follows from~\eqref{eq:54} and \eqref{eq:53} again that
\begin{align}
  \beta(F, z)  \hat\gamma_\Frob(F, z)
  &\leq \frac43 \left(d_\bP(z, \zeta) + d_\bP(\cN_F(z), \zeta) \right)
  \hat\gamma_\Frob(F, \zeta)\\
  &\leq \frac{4}{3} \left( 2^{1-2^k} + 2^{1-2^{k+1}} \right) d_\bP(w, \zeta)
  \hat\gamma_\Frob(F, \zeta)\\
  &\leq \frac{1}{10} 2^{-2^k}.\label{eq:81}
\end{align}
Besides, by Theorem~\ref{thm:probabilistic-gamma}, we have with
probability at least~$\frac34$,
\begin{equation}\label{eq:82}
  c \leq 192 n^2\delta \cdot \hat\gamma_\Frob(F, z).
\end{equation}
(Note that the computation of~$c$ involves~$n$ calls to {\sc GammaProb},
each returning a result outside the specified range with
probability at most~$\frac{1}{4n}$. So the~$n$ computations are correct
with probability at least~$\frac34$.)
It follows from~\eqref{eq:81} and~\eqref{eq:82}, along with the choice
of~$k$, that
\begin{equation}
  2^{2^{k-1}} \beta(F, z) c \leq \tfrac{192}{10} n^2 \delta 2^{-2^{k-1}} \leq \alpha_0,
\end{equation}
with probability at least~$\frac34$.
We conclude, assuming~\eqref{eq:53}, that Algorithm~{\sc BlackBoxSolve}
succeeds with probability at least~$\frac34$.

Assume now that the algorithm succeeds but~$z$, the output point, is
not an approximate zero of~$F$.  On the one hand, $z$ is not an
approximate zero, so
\begin{equation}
    \beta(F, z) \hat\gamma_\Frob(F, z) > \alpha_0,
\end{equation}
and on the other hand, the algorithm succeeds,
so~$2^{2^{k-1}} \beta(F, z) c \leq \alpha_0$, and then
\begin{equation}\label{eq:4.23}
    2^{2^{k-1}} c \leq \hat\gamma_\Frob(F, z).
  \end{equation}
By definition~\eqref{eq:41F} of~$\hat\gamma_\Frob(F, z)$, and since
$c = \kappa(F,z) (\Gamma_1^2 + \cdots +\Gamma_n^2)^{\frac12}$,
where $\Gamma_i$ denotes the value returned by the call to
$\text{\sc GammaProb}(f_i,z,\frac{1}{4n})$,
we get
\begin{equation}
  2^{2^{k-1}} (\Gamma_1^2 + \cdots +\Gamma_n^2)^{\frac12} \leq \left( \gamma_\Frob(f_1,z)^2+\dotsb+\gamma_\Frob(f_n, z)^2 \right)^{\frac12}.
\end{equation}
This implies that, for some~$i$,
\begin{equation}\label{eq:66}
   2^{2^{k-1}} \Gamma_i \leq \hat\gamma_\Frob(f_i, z).
\end{equation}
By choice of~$k$, $2^{2^{k-1}} \geq \epsilon^{-1}$, and
using the tail bound in Theorem~\ref{thm:probabilistic-gamma}, with~$t=\epsilon^{-1}$,
\eqref{eq:66} may only happen with probability at most
\begin{equation}
\frac{1}{4n}\Big(\frac{1}{4n}\Big)^{\frac12 \log_2 t} \le
 \left(\frac{1}{4}\right)^{\frac12 \log_2 t} =
 2^{- \log_2 t} = \epsilon.
\end{equation}

The complexity bound is clear since a Newton iteration requires only~$\poly(n, \delta) L(F)$ operations.
\end{proof}

\begin{algo}[tp]
  \centering
  \begin{algorithmic}
    \Function{\sc BoostBlackBoxSolve}{$F$, $\epsilon$}
    \Repeat
    \State $w\gets \textsc{BlackBoxSolve}(F, \frac14)$
    \Comment Algorithm~\ref{algo:solve}
    \State $z \gets \textsc{Boost}(F, w, \epsilon)$
    \Comment Algorithm~\ref{algo:validate}
    \Until{$z\neq\FAIL$}
    \State \Return $z$
    \EndFunction
  \end{algorithmic}
  \caption[]{Boosted zero finder for black-box input
    \begin{description}
    \item[Input:] $F \in\cH$ (given as black-box), $\epsilon \in (0,\frac12)$
      \item[Output:] $z \in \mathbb{P}^n$ if algorithm terminates.
      \item[Postcondition:] $z$ is an approximate zero of~$F$ with
        probability~$\geq 1-\epsilon$.
    \end{description}
  }
  \label{algo:validatesolve}
\end{algo}

The combination of {\sc BlackBoxSolve} and {\sc Boost} leads
to Algorithm~\ref{algo:validatesolve}, {\sc BoostBlackBoxSolve}.

\begin{proof}[Proof of Theorem~\ref{thm:first-main-result-complexity}] The correctness, with
probability at least~$1-\epsilon$, is clear, by the correctness of {\sc Boost}. An
iteration of Algorithm~\ref{algo:validatesolve} succeeds if and only if
{\sc Boost} succeeds. If~\eqref{eq:53} holds (which it does with
probability at least~$\frac34$), then {\sc Boost} succeeds with probability at least~$\frac34$. 
So each iteration of Algorithm~\ref{algo:validatesolve} succeeds
with probability at least~$\frac12$, and the expected number of iterations is therefore at
most two. Furthermore, on input~$\bfu \cdot F$, the average cost of each iteration
is~$\poly(n,\delta) L(F) \Gamma(F) \log \Gamma(F)$ for {\sc BlackBoxSolve} and
$\poly(n,\delta) L(F) \log\log \epsilon^{-1}$ for {\sc Boost}.
\end{proof}

\begin{proof}[Proof of Corollary~\ref{coro:main-result}]
Let~$\bfu\in\cU$ be uniformly distributed and independent from~$F$.
By hypothesis, $\bfu\cdot F$ and~$F$ have the same distribution, so we
study~$\bfu\cdot F$ instead.  Then Theorem~\ref{thm:first-main-result-complexity}
applies and we obtain, for fixed $F\in\cH$ and random $u\in\cU$,
that {\sc BoostBlackBoxSolve} terminates after
\begin{equation}
  \mop{poly}(n, \delta) \cdot L \cdot \left( \EE \left[ \Gamma(F)
    \log \Gamma(F) \right]  + \log \log \epsilon^{-1} \right)
\end{equation}
operations on average. With the concavity on~$[1,\infty)$ of the
function~$h:x\mapsto x^{\frac12} \log x^{\frac12}$, Jensen's inequality
ensures that
\begin{equation}
  \EE \left[ \Gamma(F)\log \Gamma(F) \right]
  = \EE \left[ h(\Gamma(F)^2) \right]
  \leq h \left( \EE[\Gamma(F)^2] \right),
\end{equation}
which gives the complexity bound.
\end{proof}

\section{Probabilistic analysis of algebraic branching programs}
\label{sec:gauss-algebr-branch}

The goal of this section is to prove our second main result, Theorem~\ref{thm:second-main}.
Recall from \S\ref{se:MainResults-II} the notion of a \GRABP.
We first state a result that connects the notions of irreducible \GRABP{s} with that of irreducible polynomials.

\begin{lemma} \label{le:GRABP-irred}
  Let~$f$ be the homogeneous polynomial computed by an irreducible \GRABP{} in the variables~$z_0,\dotsc,z_n$.
  If~$n\geq 2$ then~$f$ is almost surely irreducible.
\end{lemma}

\begin{proof}
The proof is by induction on the degree~$\delta$, the base case $\delta=1$ being clear.
So suppose $\delta\ge 2$. 
In the given ABP replace the label of each edge~$e$ by a new variable~$y_e$.
Let~$G$ denote the modified ABP and~$g$ the polynomial computed by~$G$.
The polynomial~$f$ is obtained as a restriction of~$g$ to a generic linear subspace, 
so, by Bertini's theorem, it suffices to prove that $g$ is irreducible (recall $n\ge 2$).

Let $s$ denote the source vertex and $t$ the target vertex of $G$.
There is a path from $s$ to $t$: let $e = (s,v)$ be its first edge.
We remove $s$ and all vertices in the first layer different from $v$, making 
$v$ the source vertex of a new ABP denoted $H$. It is irreducible:
if the layers of~$G$ have the sizes $1, r_1,\dotsc, r_{\delta-1}, 1$,
then the layers of $H$ have the sizes $1,r_2,\dotsc,r_{\delta-1}, 1$.
The paths of $H$ from source to target are in bijective correspondence with the paths of $G$ from $v$ to $t$.
Therefore, 
$g = y_e p + q$, where~$p$ is the polynomial computed by~$H$,
and~$q$ corresponds to the paths from~$s$ to~$t$ which avoid~$v$.
By induction hypothesis, $p$ is irreducible.
Clearly, $q\neq 0$ because~$r_1 > 0$, 
and $p$ does not divide $q$ since the variable corresponding to
an edge leaving~$v$ does not appear in $q$ (such edge exists due to $\delta\ge 2$). 
We conclude that~$p$ and~$q$ are relatively prime.
Moreover, the variable~$y_e$ does neither appear in~$p$ nor in~$q$,
so it follows that $g$ is irreducible.
\end{proof}

We also remark that a random polynomial computed by a \GRABP
may define a random hypersurface in~$\mathbb{P}^n$ that is always singular.
It is rather uncommon in our field to be able to study stochastic models featuring singularities almost surely, so it is worth a lemma.

\begin{lemma}\label{lem:singular-abp}
  If~$f \in \mathbb{C}[z_0,\dotsc,z_n]$ is the polynomial computed by a \ABP with at most~$n$ edges,
  then the hypersurface~$V(f) \subset \mathbb{P}^n$ is singular.
\end{lemma}

\begin{proof}
  Let~$e$ be the number of edges of the \ABP computing~$f$.
  After a linear change of variables, we may assume that~$f$ depends only 
on~$z_0,\dotsc,z_{e-1}$.
  The singular locus of~$V(f)$ is defined by the vanishing of the partial 
derivatives~$\frac{\partial}{\partial z_i} f$.
  But these derivatives are identically~0 for~$i\geq e$, so that the singular locus is defined by at most~$e$ equations.
  So it is nonempty.
\end{proof}

As already mentioned before,
the distribution of a polynomial computed by a \GRABP is best 
understood in terms of matrices.  
This calls for the introduction of some terminology.
For any $\delta$-tuple $\bfr =(r_1,\dotsc,r_\delta)$, let $M_\bfr(n+1)$ 
(and $M_\bfr$ for short)
denote the space of all $\delta$-tuples
of matrices~$(A_1(z),\dotsc,A_\delta(z))$, of respective
size~$r_\delta\times r_1$, $r_1\times r_2$, \ldots,
$r_{\delta-1}\times r_\delta$, with degree one homogeneous entries in $z=(z_0,\ldots,z_n)$.
(It is convenient to think of $r_0=r_\delta$.) We have
$\dim_{\bC}M_{\bfr}=(n+1)\sum_{i=1}^\delta r_{i-1}r_i$.
For~$A \in M_\bfr$, we define the degree $\delta$ homogeneous polynomial
\begin{equation}\label{eq:def-f_A}
  f_A(z) \eqdef \tr \left( A_1(z) \dotsb A_\delta(z) \right) .
\end{equation}
A Hermitian norm is defined on~$M_\bfr$ by
\begin{equation*}
  \|A\|^2 \eqdef \sum_{i=1}^\delta \sum_{j=0}^n \|A_i(e_j)\|^2_\Frob,
\end{equation*}
where~$e_j = (0,\dotsc,0,1,0,\dotsc,0)\in\bC^{n+1}$, with a~1 at index~$j$
($0\leq j \leq n$).
The standard Gaussian probability on~$M_{\bfr}$ is defined by the
density~$\pi^{-\dim_{\bC} M_{\bfr}} \exp(-\|A\|^2) \ud A$.  
The distribution of the polynomial computed by a \GRABP with layer
sizes~$(r_1,\dotsc,r_{\delta-1})$ is the distribution of~$f_A$, 
where~$A$ is standard Gaussian in~$M_{(r_1,\dotsc,r_{\delta-1},1)}$.

The following statement is the main ingredient of the
proof of Theorem~\ref{thm:second-main}.
It can be seen as an analogue of Lemma~I.37. 
(Note that $r_\delta=1$, the case of interest of ABPs, is included.)

\begin{proposition}\label{prop:main-prop-stochastic}
Assume that~$r_1,\dotsc,r_{\delta-1}\geq 2$.
Let~$A \in M_{\bfr}$ be standard Gaussian and let~$\zeta \in\mathbb{P}^n$
be a uniformly distributed projective zero of~$f_A$. For any~$k \geq 2$, we have 
\begin{align*}
  \EE_{A,\zeta} \left[ \left\| \ud_\zeta f_A \right\|^{-2}
    \left\|{\frac{1}{k!}} \ud_\zeta^k f_A \right\|_{{\Frob}}^2 \right]
  &\leq \frac{1}{n \delta} \binom{\delta}{k}
  \binom{\delta + n}{k} \left( 1+\frac{\delta-1}{k-1} \right)^{k-1} \\
  &\leq \left[\tfrac14 \delta^2 (\delta+n)
    \left( 1+\frac{\delta-1}{k-1} \right)\right]^{k-1}.
\end{align*}
\end{proposition}

Theorem~\ref{thm:second-main} easily follows from 
Proposition~\ref{prop:main-prop-stochastic}.

\begin{proof}[Proof of Theorem~\ref{thm:second-main}]
Let~$A \in M_\bfr$ be standard Gaussian so that~$f =f_A$.
The proof follows exactly the lines of the proof of
Lemma~I.38 and the intermediate Lemma~I.37.  We bound the supremum
in the definition~\eqref{eq:4} of~$\gamma_\Frob$ by a sum:
\begin{align} 
  \bE \left[ \gamma_\Frob(f_A,\zeta)^2 \right]
  &\leq \sum_{k = 2}^\delta
  \bE \left[ \left( \left\| \ud_\zeta f_A \right\|^{-1}
    \left\| \tfrac{1}{k!} \ud^k_\zeta f_A \right\|_\Frob \right)^{\frac{2}{k-1}} \right] \\
  &\leq  \sum_{k = 2}^\delta \bE \left[\left\| \ud_\zeta f_A \right\|^{-2}
    \left\| \tfrac{1}{k!} \ud^k_\zeta f_A \right\|_\Frob^{2}
    \right]^{\frac{1}{k-1}}\label{eq:59} \\
  &\leq \sum_{k=2}^\delta  \tfrac14 \delta^2 (\delta+n)
  \left( 1+\frac{\delta-1}{k-1} \right),
  \text{ by Proposition~\ref{prop:main-prop-stochastic},} \\
  &\leq \tfrac34 \delta^3 (\delta+n)\log \delta, \label{eq:60}
\end{align}
using Jensen's inequality for~\eqref{eq:59} and 
$1 +\sum_{k=2}^\delta \frac{1}{k-1} \leq 2 +\log(\delta-1) \le  3 \log \delta$
for~\eqref{eq:60}.
\end{proof}

The remaining of this article is devoted to the proof of
Proposition~\ref{prop:main-prop-stochastic}.

\subsection{A coarea formula}\label{se:coarea}

The goal of this subsection is to establish a consequence of the coarea formula 
\parencite[Theorem~3.1]{Federer_1959} that is especially
useful to estimate~$\Gamma(f)$ for a random polynomial~$f$. 
This involves a certain identity of normal Jacobians of projections 
that appears so frequently that it is worthwhile to provide the statement 
 
in some generality.

Let us first introduce some useful notations. 
For a linear map~$h : E \to F$ between two Euclidean spaces
we define its \emph{Euclidean determinant} as
\begin{equation}
  \mop{Edet}(h) \eqdef \det(h\circ h^{\mathsf t})^{\frac12},
  \label{eq:57}
\end{equation}
where~$h^{\mathsf t} : F\to E$ is the transpose of~$h$.
If~$p : U\to V$ is a linear map between Hermitian spaces,
then~$\mop{Edet}(p)$ is defined by the induced Euclidean structures
on~$U$ and~$V$ and it is well known that
\begin{equation}
  \mop{Edet}(p) = \det(p \circ p^*),
\end{equation}
where~$p^* : V \to U$ is the Hermitian transpose (and~$\det$ is the
determinant over~$\mathbb{C}$). 

The {\em normal Jacobian} of a smooth map~$\phi$ between Riemannian manifolds at a given point~$x$ 
is defined as the Euclidean determinant of the derivative of the map at that point:
\begin{equation}
  \mop{NJ}_x \phi \eqdef \mop{Edet} \left( \ud_x \phi \right).
\end{equation}

\begin{lemma}\label{lem:quotien-of-NJ}
  Let~$E$ and~$F$ be Euclidean (resp.~Hermitian) spaces, let~$V$ be a subspace
  of~$E\times F$ and let~$p:E\times F \to E$ and~$q:E\times F\to F$ be the
  canonical projections. Then $\mop{Edet}(p|_V) = \mop{Edet}(q|_{V^\perp})$ and
  $\mop{Edet}(q|_V) = \mop{Edet}(p|_{V^\perp})$.
\end{lemma}

\begin{proof} 
By symmetry, it suffices to show the first equality. 
Let $v_1,\dotsc,v_r, w_1,\dotsc,w_s$ be an orthonormal basis of~$E\times F$
such that $v_1,\dotsc,v_r$ is a basis of $V$ and 
$w_1,\dotsc,w_s$ is basis of~$V^\perp$. 
After fixing orthonormal bases for~$E$ and~$F$ (and the corresponding
basis of~$E\times F$), consider the orthogonal (resp. unitary) 
matrix $U$ with the columns $v_1,\dotsc,v_r, w_1,\dotsc,w_s$. 
We decompose $U$ as a block matrix
\begin{equation}
  \setlength{\arrayrulewidth}{1pt}
  U \eqdef
  \left[
    \begin{array}{c|c}
      V_E & W_E \\ \hline
      V_F & W_F
    \end{array}
  \right] \eqdef
  \left[
    \begin{array}{c!{\vrule width .5pt}c!{\vrule width .5pt}c|c!
        {\vrule width .5pt}c!{\vrule width .5pt}c}
      p(v_1) & \dotsc & p(v_r) & p(w_1) & \dotsc & p(w_s) \\ \hline
      q(v_1) & \dotsc & q(v_r) & q(w_1) & \dotsc & q(w_s)
    \end{array} \right].
\end{equation}
Using $U U^*=I$ and $U^* U = I$ we see that 
$V_E^{\phantom{*}} V_E^* + W_E^{\phantom{*}} W_E^* = I$ 
and 
$W^*_E W^{\phantom{*}}_E + W_F^* W_F^{\phantom{*}} = I$.
It follows from Sylvester's determinant identity~$\det(I+AB)=\det(I+BA)$ that
\begin{equation}
  \det(V_E^{\phantom{*}} V_E^*) = \det( I- W_E^{\phantom{*}} W_E^*)
  = \det( I- W_E^* W_E^{\phantom{*}} ) = \det(W_F^* W_F^{\phantom{*}}) .\label{eq:79}
\end{equation}
By definition, we have  
$\mop{Edet}(p|_V)^\eta = \det(V_E^{\phantom{*}} V^*_E )$ 
with $\eta=1$ in the Euclidean situation and $\eta=2$ in the 
Hermitian situation. Similarly, 
$\mop{Edet}(q|_{V^\perp})^\eta = \det(W_F^{\phantom{*}} W^*_F )$.
Therefore, indeed $\mop{Edet}(p|_V) = \mop{Edet}(q|_{V^\perp})$.
\end{proof}

\begin{corollary}\label{cor:NJ-quot}
In the setting of Lemma~\ref{lem:quotien-of-NJ}, suppose $V$
is a real (or complex) hyperplane in~$E\times F$
with nonzero normal vector $(v,w)\in E \times F$. Then 
\[
 \frac{\mop{Edet}(p|_{V})}{\mop{Edet}(q|_{V})} = \left(\frac{\|w\|}{\|v\|}\right)^\eta,
\]
where $\eta=1$ in the Euclidean situation and $\eta=2$ in the Hermitian situation.
\end{corollary}

\begin{proof}
$V^\perp$ is spanned by $(v,w)$ and therefore, 
$\mop{Edet}(p|_{V^\perp})= \Big(\|v\| /\sqrt{\|v\|^2+ \|w\|^2}\Big)^{\eta}$, 
and 
$\mop{Edet}(q|_{V^\perp})= \Big(\|w\| /\sqrt{\|v\|^2+ \|w\|^2}\Big)^{\eta}$.
Now apply  Lemma~\ref{lem:quotien-of-NJ}.
\end{proof}

We consider now the abstract setting of
a family $(f_A)$ of homogeneous polynomials of degree~$\delta$
in the variables~$z_0,\dotsc,z_n$, 
parameterized by elements~$A$ of a Hermitian manifold~$M$  
through a holomorphic map~$A \in M\mapsto f_A$.
Let~$\cV$ be the solution
variety~$\left\{ (A,\zeta)\in M\times\Proj \st f_A(\zeta) = 0 \right\}$ 

and $\pi_1 : \cV\to M$ and~$\pi_2 : \cV \to \bP^n$ be the restrictions of
the canonical projections. 
We can identify the fiber $\pi_1^{-1}(A)$ with the zero set $V(f_A)$ in $\Proj$.  
Moreover, the fiber $\pi_2^{-1}(\zeta)$ can be identified with 
$M_\zeta \eqdef \left\{ A\in M\st f_A(\zeta) = 0 \right\}$.
For fixed $\zeta\in\Proj$, we consider the map $M\to\bC,\, A \mapsto f_A(\zeta)$
and its derivative at $A$,
\begin{equation}
  \partial_A f(\zeta) : T_A M \to \bC .\label{eq:70}
\end{equation}
Moreover, for fixed $A\in M$, we consider the map 
$f_A\colon\bC^{n+1}\to\bC$ and its derivative at $\zeta$, 
\begin{equation}
  \ud_\zeta f_A : T_\zeta\Proj \to \bC ,\label{eq:71}
\end{equation}
restricted to the tangent space $T_\zeta\Proj$, that we identify
with the orthogonal complement of $\bC\zeta$ in $\bC^{n+1}$ 
with respect to the standard Hermitian inner product.

\begin{proposition}\label{prop:rice}
For any measurable function~$\Theta: \cV\to [0,\infty)$, we have 
\[ \smashint[\ud A]{M} \smashint[\ud\zeta]{V(f_A)}\ \Theta(A, \zeta) \| \partial_A f(\zeta) \|^{2} 
  =  \smashint[\ud\zeta]{\Proj} \smashint[\ud A]{M_\zeta} \
  \Theta(A, \zeta) \|\ud_\zeta f_A\|^{2} .
\]
Here $\ud A$ denotes the Riemannian volume measure on $M$ and $M_\zeta$, respectively.
\end{proposition}

\begin{proof}
As in \parencite[Lemma~16.9]{BurgisserCucker_2013},
the tangent space of $\cV$ at $(A,\zeta) \in \cV$ can be expressed as 
\begin{equation}\label{eq:1}
  V \eqdef T_{A,\zeta} \cV = \left\{ (\dot A, \dot\zeta)\in T_A M\times 
T_\zeta\Proj
  \st \ud_\zeta f_A(\dot\zeta) + \partial_A f(\zeta)(\dot A) = 0 \right\}.
\end{equation}
If $\partial_A f(\zeta)$ and $\ud_\zeta f_A$ are not both zero, then 
$V$ is a hyperplane in the product 
$E\times F \eqdef T_A M \times T_\zeta\Proj$ of Hermitian spaces
and $V$ has the normal vector $(\partial_A f(\zeta), \ud_\zeta f_A)$, 
upon identification of spaces with their duals.
If we denote by $p$ and $q$ the canonical projections of $V$ onto 
$E$ and $F$, then 
$\ud_{A,\zeta} \pi_1 = p|_V$ and $\ud_{A,\zeta} \pi_2 = q|_V $, 
hence 
\begin{equation}
 \mop{NJ}_{A,\zeta}(\pi_1) = \mop{Edet}(p|_V), \quad \mop{NJ}_{A,\zeta}(\pi_2) = \mop{Edet}(q|_V) .\label{eq:69}
\end{equation}
By Corollary~\ref{cor:NJ-quot}, we therefore have
\begin{equation}\label{eq:NJ-Ident}
 \frac{\mop{NJ}_{A,\zeta}(\pi_1)}{\mop{NJ}_{A,\zeta}(\pi_2)} 
 = \frac{\mop{Edet}(p|_V)}{\mop{Edet}(q|_V)} = \frac{\|\ud_\zeta f_A\|^2}{\|\partial_A f(\zeta)\|^2}  .
\end{equation}

The coarea formula \parencite[Theorem~3.1]{Federer_1959} applied to~$\pi_1 : \mathcal{V}\to M$ asserts,
\begin{equation}\label{eq:2}
  \smashint[\ud (A,\zeta)]{\cV}  \Theta(A, \zeta) \|\ud_\zeta f_A\|^2
  \mop{NJ}_{A,\zeta}(\pi_1) = \smashint[\ud A]{M} \smashint[\ud\zeta]{V(f_A)}\
  \Theta(A, \zeta) \|\ud_\zeta f_A\|^2 .
\end{equation}
(Note that $\cV$ may have singularities, so we actually apply the coarea formula to its smooth locus.)
On the other hand, the coarea formula applied to $\pi_2:\cV \to \mathbb{P}^n$ gives
\begin{equation}\label{eq:3}
  \smashint[\ud (A,\zeta)]{\cV}  \Theta(A, \zeta)
  \|\partial_A f(\zeta)\|^2 \mop{NJ}_{A,\zeta}(\pi_2)
  = \smashint[\ud\zeta]{\Proj} \smashint[\ud A]{M_\zeta} \
  \Theta(A, \zeta) \|\partial_A f(\zeta)\|^2.
\end{equation}
By \eqref{eq:NJ-Ident} we have 
\begin{equation}
\mop{NJ}_{A,\zeta}(p) \|\partial_A f(\zeta)\|^2 = \mop{NJ}_{A,\zeta}(q) 
\|\ud_\zeta f_A\|^2 ,\label{eq:68}
\end{equation}
so all the four integrals above
are equal.
\end{proof}

\subsection{A few lemmas on Gaussian random matrices}\label{se:few:LGRM}

We present here some auxiliary results on Gaussian random matrices, 
centering around the new notion of the {\em anomaly} of a matrix.
This will be crucial  for the proof of Theorem~\ref{thm:second-main}. 

We endow the space~$\mathbb{C}^r$ with the probability
density~$\pi^{-r} e^{-\|x\|^2} \ud x$, where~$\|x\|$ is
the usual Hermitian norm, and call a random vector~$x\in \mathbb{C}^r$ 
with this probability distribution \emph{standard Gaussian}.
{This amounts to say that the real and imaginary parts of $x$
are independent centered Gaussian with variance $\frac12$. 
Note that $\EE_x \left[\|x\|^2\right] = r$.}
This convention slightly differs from some previous writings 
with a different scaling, where the distribution used is~$(2\pi)^{-r} e^{-\frac12\|x\|^2} \ud x$. 
This choice seems more natural since it avoids many spurious factors.
Similarly, the matrix space~$\mathbb{C}^{r\times s}$ is endowed with the
probability density $\pi^{-rs} \exp(-\|R\|_\Frob^2) \ud R$,
and we call a random matrix with this probability distribution 
\emph{standard Gaussian} as well. (In the random matrix literature 
this is called complex Ginibre ensemble.)

\begin{lemma}\label{le:ThetaLB}
For  $P\in \mathbb{C}^{r\times s}$ fixed and $x \in \mathbb{C}^s$
standard Gaussian, we have 
$$
  \EE_x \left[\|P x\|^2\right] = \|P\|_\Frob^2 , \quad 
  \EE_x \left[\|P x\|^{-2}\right] \ge\|P\|_\Frob^{-2}  ,\quad \EE_x \left[\|x\|^{-2}\right] = \frac{1}{s-1} .
$$
\end{lemma}

\begin{proof}
By the singular value decomposition and unitary invariance, we may assume 
that
$P$ equals $\mathrm{diag}(\sigma_1,\ldots,\sigma_{\min(r, s)})$, with zero columns or zero rows appended. Then
$\|Px\|^2 = \sum_i \sigma_i^2 |x_i|^2 $, hence
$\EE_x \left[\|P x\|^2\right] = \sum_i \sigma_i^2 \EE_{x_i} \left[\|x_i\|^2\right] = \sum_i \sigma_i^2 = \|P\|_\Frob^2$. 

For the second assertion, we note that for a nonnegative random variable~$Z$, we have 
by Jensen's inequality that 
$\EE\left[Z\right]^{-1} \le \EE\left[Z^{-1}\right]$, 
since $x\mapsto x^{-1}$ is convex on $(0,\infty)$. 
The second assertion follows by 
applying this to $Z \eqdef \|Px\|^2$ and using the first assertion. 

For the third assertion, we note $\|x\|^2 = \frac12 \chi_{2s}^2$, where 

$\chi_{2s}^2$ stands for a chi-square distribution with $2s$ degrees of freedom. 
It is known that $\EE [ \chi_{2s}^{-2} ] = 1/(2s-2)$.
\end{proof}

We define the  \emph{anomaly} 
of a matrix~$P \in \mathbb{C}^{r\times s}$ 
as the quantity
\begin{equation}\label{eq:def-anomaly}
  \theta(P) \eqdef \EE_x \left[ \frac{\|P\|_\Frob^2}{\|P x\|^2} \right] \in [1,\infty) ,
\end{equation}
where~$x \in \mathbb{C}^s$ is a standard Gaussian random vector.
Note that $\theta(P) \ge 1$ by Lemma~\ref{le:ThetaLB}. Moreover,
by the same lemma, 
$\theta(I_r) = r/(r-1)$. 
This quantity $\theta(P)$ is easily seen to be finite if $\mop{rk} P > 1$;
it grows logarithmically {to infinity}
as~$P$ approaches a rank~1 matrix.

\begin{lemma}\label{lem:anomaly}
Let~$P \in \mathbb{C}^{r\times s}$ and~$Q \in \mathbb{C}^{t\times u}$ be fixed
matrices and~$X \in {\mathbb{C}^{s\times t}}$ be a 
standard Gaussian random matrix. Then
\[ \EE_X \left[ \frac{\|P\|_\Frob^2 \|Q\|_\Frob^2}{\|P X Q\|_\Frob^2} \right] \leq \theta(P). \]
\end{lemma}

\begin{proof}
Up to left and right multiplications of~$Q$ by unitary matrices, we
may assume that~$Q$ is diagonal, with nonnegative real
numbers~$\sigma_1,\dotsc,\sigma_{\min(t, u)}$ on the diagonal
(and we define~$\sigma_ i=0$ for $i > \min(t, u)$). 
This does not
change the left-hand side because the Frobenius norm is invariant by
left and right multiplications with unitary matrices, and the
distribution of~$X$ is unitary invariant as well. 

Let~$e^u_1,\dotsc,e^u_u$ (reps.~$e_1^t, \dotsc, e_t^t)$ be the canonical basis of~$\mathbb{C}^u$ (resp.~$\mathbb{C}^t$).
Observe that
\begin{equation}\label{eq:11}
  \|PXQ\|_\Frob^2 = \sum_{i=1}^u \|PXQ e^u_i\|^2 = \sum_{i=1}^{t} 
\sigma_i^2 \|PX e^{t}_i \|^2.
\end{equation}
Noting that~$\|Q\|_\Frob^2 = \sigma_1^2 + \dotsb + \sigma_{t}^2$,
the convexity of $x \mapsto x^{-1}$ on~$(0,\infty)$ gives
\begin{equation}\label{eq:10}
  \left( \frac{1}{\|Q\|_\Frob^2} \sum_{i=1}^{t} \sigma_i^2 \|P X e^{t}_i \|^2 \right)^{-1}
  \leq \frac{1}{\|Q\|_\Frob^2} \sum_{i=1}^{t} \frac{\sigma_i^2}{\|P Xe^{t}_i\|^2}.
\end{equation}
Since~$X$ is standard Gaussian, $X e^{t}_i \in \mathbb{C}^s$ is also standard
Gaussian. Therefore, by definition of~$\theta$, we have for any~$1\leq i\leq t$,
\begin{equation}\label{eq:9}
  \EE_X \left[ \frac{\|P\|_\Frob^2}{\|P X e^{t}_i\|^2} \right] = \theta(P).
\end{equation}
It follows that
\begin{align*}
  \label{eq:8}
  \EE_X \left[ \frac{\|P\|_\Frob^2 \|Q\|_\Frob^2}{\|P X Q\|_\Frob^2} \right]
  &\leq \frac{1}{\|Q\|_\Frob^2} \sum_{i=1}^t \sigma_i^2
  \EE \left[ \frac{\|P\|_\Frob^2}{\|P Xe^{t}_i\|^2} \right],
  && \text{by~\eqref{eq:11} and~\eqref{eq:10},}\\
  &= \frac{1}{\|Q\|_\Frob^2} \sum_{i=1}^t \sigma_i^2 \theta(P),
  &&\text{by \eqref{eq:9},} \\
  & = \theta(P),
\end{align*}
which concludes the proof.
\end{proof}

\begin{lemma}\label{lem:anomaly-product}
Let~$P \in \mathbb{C}^{r\times s}$ be fixed, $t>1$, and~$X\in \mathbb{C}^{s\times t}$ be
a standard Gaussian random matrix. Then
\[
 \EE_X \left[ \theta(PX) \right] = \frac{1}{t-1} + \theta(P).
\]
Furthermore, if~$X_1,\dotsc,X_m$ are standard
Gaussian matrices of size~$r_{0}\times r_1$, $r_1\times r_2, \ldots, r_{m-1}\times r_m$,
respectively, where $r_0,\ldots,r_m >1$, then
\[ \EE_{X_1,\ldots,X_m} \left[ \theta(X_1 \dotsb X_m ) \right]
= 1 + \sum_{i=0}^m \frac{1}{r_i-1}. \]
\end{lemma}

\begin{proof}
Let~$x \in \mathbb{C}^t$ be a standard Gaussian random vector, so that
\begin{equation}
  \EE_X \left[ \theta(PX) \right]
  = \EE_{X,x} \left[ \frac{\|P X\|_\Frob^2}{\|P X x\|^2} \right].
\end{equation}
We first compute the expectation conditionally on~$x$. 
So we fix $x$ and write~$x = \|x\| u_1$ for some unit vector~$u_1$.
We choose other unit vectors~$u_2,\dotsc,u_t$
to form an orthonormal basis of~$\mathbb{C}^t$.
Since~$\|PX\|_\Frob^2 = \sum_{i=1}^t \|PXu_i\|^2$, we obtain
\begin{equation}\label{eq:12}
  \frac{\|PX\|_\Frob^2}{\|PXx\|^2}
  = \frac{1}{\|x\|^2} + \sum_{i=2}^t \frac{\|PX u_i\|^2}{\|PXu_1\|^2 \|x\|^2}.
\end{equation}
Since~$X$ is standard Gaussian, the vectors~$X u_i$ are standard Gaussian 
and
independent. So we obtain, using Lemma~\ref{le:ThetaLB}, 
\begin{align}
  \EE_X \left[ \frac{\|PX u_i\|^2}{\|PXu_1\|^2}  \right]
  &= \EE_X [ \|PXu_i\|^2 ] \, \EE_X \left[ \frac{1}{\|PX u_1\|^2} \right]   \\
  &= \|P\|_\Frob^2 \, \EE_X \left[ \frac{1}{\|PX u_1\|^2} \right] = \theta(P)  .\label{eq:72}
\end{align}
Combining with~\eqref{eq:12}, we obtain
\begin{equation}
  \EE_{X} \left[ \frac{\|P X\|_\Frob^2}{\|P X x\|^2} \right]
  = \frac{1}{\|x\|^2} + \sum_{i=2}^t \frac{\theta(P)}{\|x\|^2} = \frac{1}{\|x\|^2} \big(1+ (t-1) \theta(P)\big). \label{eq:73}
\end{equation}
When we take the expectation over $x$, the
first claim follows with the third statement of Lemma~\ref{le:ThetaLB}.

The second claim follows by induction on~$m$. 
The base case~$m=1$ follows from
writing $\EE_{X_1} \left[ \theta(X_1) \right]=\EE_{X_1} \left[ \theta(I_{r_0}X_1) \right]$, 
the first part of Lemma~\ref{le:ThetaLB}, and 
$\theta( I_{r_0} ) = 1+\frac{1}{r_0-1}$.
For the induction step $m>1$, we first fix $X_1,\ldots,X_{m-1}$ and obtain from the first assertion
\begin{equation}
  \EE_{X_{m}} \left[ \theta(X_1 \dotsb X_{m-1} X_m ) \right] =\frac{1}{r_m-1} + \theta(X_1 \dotsb X_{m-1}).\label{eq:76}
\end{equation}
Taking the expectation over $X_1,\ldots,\dotsb, X_{m-1}$ and applying the 
induction hypothesis 
implies the claim. 
\end{proof}

\begin{lemma}\label{lem:int-tr}
For any fixed~$P, Q \in \mathbb{C}^{r\times r}$
and~$X \in \mathbb{C}^{r\times r}$ standard Gaussian, we have 
\begin{enumerate}[(i)]
\item\label{item:int-tr:1}
  $\mathbb{E} \left[ \left| \tr (XQ) \right|^2 \right] = \left\| Q \right\|_\Frob^2$, 
\item\label{item:int-tr:2}
  $\mathbb{E} \left[ \left\| PXQ \right\|_\Frob^2 \right] = \left\| P \right\|_\Frob^2
  \left\| Q \right\|_\Frob^2$.
\end{enumerate}
\end{lemma}

\begin{proof}
By unitarily invariance of the distribution of~$X$ and the Frobenius norm,
we can assume that~$P$ and~$Q$ are diagonal matrices. Then the claims
reduce to easy computations.
\end{proof}

\subsection{Proof of Proposition~\ref{prop:main-prop-stochastic}}
We now {carry out} the estimation of
\begin{equation}
  \EE \left[ \left\| \ud_{\zeta} f_{A} \right\|^{-2}
    \left\|{\tfrac{1}{k!}}\ud_{\zeta}^{k}f_{A}\right\|_\Frob^{2} \right],
\end{equation}
where~$A \in M_\bfr$ is standard Gaussian and~$\zeta \in \mathbb{P}^n$ is 
a
uniformly distributed zero of~$f_{A}$.
The computation is lengthy but the different ingredients arrange elegantly.

\subsubsection{Conditioning~$A$ on~$\zeta$}
As often in this kind of average analysis, the first step is to
consider the conditional distribution of~$A$ given~$\zeta$, reversing
the natural definition where~$\zeta$ is defined conditionally on~$A$. 
This is of course the main purpose of
Proposition~\ref{prop:rice}. \
Consider the Hermitian vector space $M \eqdef M_{\bfr}$ and
let~$\ud' A = \pi^{-\dim_{\mathbb{C}} M_\bfr} e^{-\sum_i \left\| A_i \right\|^2} \ud A$ 
denote the Gaussian probability measure on~$M_\bfr$. 
It is a classical fact \parencite[e.g.,][p.~20]{Howard_1993}
that the volume of a hypersurface of degree~$\delta$ in~$\mathbb{P}^n$ 
equals~$\delta \vol \mathbb{P}^{n-1}$; this applies in particular to~$V(f_{A})$.
By Proposition~\ref{prop:rice}, we have
\begin{align}
 \MoveEqLeft\notag
  \mathbb{E} \left[ \left\| \ud_{\zeta} f_{A} \right\|^{-2}
    \left\|{\tfrac{1}{k!}}\ud_{\zeta}^{k} f_{A}\right\|_\Frob^{2}\right]\\
  &= \smashint[\ud' A]{M}\; (\vol V(f_{A}))^{-1}\smashint[\ud\zeta]{V(f_A)}\
  \left\| \ud_{\zeta} f_{A} \right\|^{-2}
  \left\|{\tfrac{1}{k!}}\ud_{\zeta}^{k} f_{A}\right\|_\Frob^{2} \\
\label{eq:27}  &=  \left( \delta \vol \mathbb{P}^{n-1} \right)^{-1}
\smashint[\ud\zeta]{\Proj} \smashint[\ud' A]{M_\zeta}\
\left\|\partial_{A} f(\zeta) \right\|^{-2}
\left\|{\tfrac{1}{k!}}\ud_{\zeta}^{k} f_{A}\right\|_\Frob^{2}.
\end{align}
{Here $\ud'A$ denotes the Gaussian measure on $M$ and $M_\zeta$, respectively.}

We focus on the inner integral over~$M_{\zeta}$ for some fixed~$\zeta$. 
Everything being unitarily invariant, this integral actually does not 
depend on~$\zeta$. So we fix~$\zeta \eqdef [1:0:\dotsb:0]$. 
We next note that~$\vol \mathbb{P}^n =\frac{\pi}n\vol\mathbb{P}^{n-1}$
and we obtain
\begin{equation}\label{eq:20}
  \mathbb{E} \left[ \left\| \ud_{\zeta} f_{A} \right\|^{-2}
    \left\| \tfrac{1}{k!} \ud_{\zeta}^{k} f_{A}\right\|_\Frob^{2}\right]
  = \frac{\pi}{\delta n}
  \smashint[\ud' A]{M_\zeta}\  \left\|\partial_{A} f(\zeta) \right\|^{-2}
  \left\|{\tfrac{1}{k!}}\ud_{\zeta}^{k} f_{A}\right\|_\Frob^{2}.
\end{equation}

Recall that the entries of  $A_{i} = A_i(z)$ are linear forms in
$z_0,z_{1},\dotsc,z_{n}$. 
We define 
\begin{equation}\label{eq:decomp}
 B_i \eqdef A_i(\zeta) \in\bC^{r_{i-1}\times r_i},\quad 
  A_i(z) = z_{0} B_{i} + C_{i}(z_{1},\dotsc,z_{n}), 
\end{equation}
where the entries of the matrix $C_i(z_{1},\dotsc,z_{n})$ are linear forms
$z_{1},\dotsc,z_{n}$. 
This yields an orthogonal decomposition
$M_{\bfr}(n+1)\simeq M_{\bfr}(1)\oplus M_{\bfr}(n)$ 
with respect to the Hermitian norm on~$M_\bfr$, where
$A=B+C$ with 
\begin{equation}
B=(B_1,\ldots,B_{\delta})\in\prod_{i=1}^{\delta}\bC^{r_{i-1}\times r_i}\simeq M_{\bfr}(1) ,\quad 
C=(C_1,\ldots,C_\delta)\in M_{\bfr}(n).
\end{equation}

Consider the function $f(\zeta):M_{\bfr}(n+1) \to\bC$, $A\mapsto f_A(\zeta)$.
By~\eqref{eq:def-f_A} we have 
$f_A(\zeta) =\tr(A_1(\zeta),\ldots,A_\delta(\zeta))= \tr(B_1\dotsb B_\delta)$.
The derivative of~$f(\zeta)$ is given by
\begin{align}
  \partial_A f(\zeta) (\dot A)
  = \sum_{i=1}^\delta \tr \left( B_1 \dotsb B_{i-1} \dot B_i B_{i+1}
  \dotsb B_\delta \right)
  = \sum_{i=1}^\delta \tr( \dot B_i \hat B_i), \label{eq:58}
\end{align}
where $\dot{A}=\dot{B}+\dot{C}$ and
(invariance of the trace under cylic permutations) 
\begin{equation}\label{eq:defBi}
\hat B_i \eqdef B_{i+1} \dotsb B_\delta \, B_1 \dotsb B_{i-1} 
\end{equation}
Hence the induced norm of the linear form $\partial_A f(\zeta)$ on 
the Hermitian space~$M_\bfr$ satisfies 
\begin{equation}\label{eq:28}
  \| \partial_A f(\zeta) \|^2 = \sum_{i=1}^\delta \|\hat B_i\|^2_\Frob.
\end{equation}

The equation defining the fiber~$M_\zeta$ can be written as
$\tr \left( B_{1} \dotsb B_{\delta} \right) = 0$.
We have $M_\zeta \simeq W \times M_\bfr(n)$, where
$W$ denotes the space of~$\delta$-tuples of complex 
matrices (of respective size~$r_0\times r_1$, $r_1\times r_2$, etc.)
that satisfy this condition. 
Using this identification,  the projection
\begin{equation}
 M_\zeta \to W,\; (A_1(z),\ldots,A_\delta(z)) \mapsto (B_1,\ldots,B_\delta) = (A_1(\zeta),\ldots,A_\delta(\zeta))
\end{equation}
is given by evaluation at~$\zeta$.
With~\eqref{eq:28},
this implies that
\begin{align}
  \MoveEqLeft \smashint[\ud'A]{M_\zeta} \  \left\|\partial_{A} f(\zeta) \right\|^{-2}
  \left\|{\tfrac{1}{k!}}\ud_{\zeta}^{k} f_{A}\right\|_\Frob^{2}
  = \smashint[\ud' B]{W} \smashint[\ud' C]{M_{\bfr} (n)} \  \left\|\partial_{A} f(\zeta) \right\|^{-2}
  \left\|{\tfrac{1}{k!}}\ud_{\zeta}^{k} f_{A}\right\|_\Frob^{2} \\
  \label{eq:17}  &= \int_W \frac{\ud' B}{\|\hat B_1\|^2 + \dotsb + \|\hat B_\delta\|^2}
                   \smashint[\ud' C]{M_{\bfr} (n)} \left\|{\tfrac{1}{k!}}\ud_{\zeta}^{k} f_{A}\right\|_\Frob^{2}.
\end{align}
As before, we denote by $\ud'B$ and $\ud'C$ the Gaussian probability measures
on the respective spaces.

\subsubsection{Computation of the inner integral}
We now study~$ \| \ud_{\zeta}^{k} f_A \|_\Frob^2$ to obtain an expression 
for
the integral~$\int \ud' C \|{\tfrac{1}{k!}}\ud_{\zeta}^{k} f_{A}\|_\Frob^{2}$
that appears in~\eqref{eq:17}. The goal is Equation~\eqref{eq:22}.

Recall that $\zeta =(1,0,\ldots,0)$. 
Let~$g(z) \eqdef f_A(\zeta + z)$ and write ~$g_{k}$ for the $k$th homogeneous component
of~$g$. By Lemma~I.30, we have 
\begin{equation}\label{eq:16}
  \left\| \tfrac{1}{k!}\ud_{\zeta}^k {f_A}\right\|_\Frob = \left\| g_{k} \right\|_{W}.
\end{equation}
By expanding a multilinear product, we compute 
{with \eqref{eq:decomp}} that
\begin{align}
  g(z_{0},\dotsc,z_{n}) &= \tr \left( \left( (1+z_0)B_1 + C_1 \right)
  \dotsb \left( (1+z_0)B_\delta + C_\delta \right) \right)\\
  &= \sum_{I\subseteq \{1,\dotsc,\delta\}} (1+z_{0})^{\delta-\# I} h_I(z_1,\dotsc,z_n),
\end{align}
{where
$h_I(z_1,\dotsc,z_n) \eqdef \tr \left( U^{I}_{1} \dotsb U^{I}_{\delta} \right)$
with 
\begin{equation}\label{def:U}
U_{i}^{I} \eqdef \begin{cases} C_{i}(z_{1},\dotsc,z_{n}) & \mbox{if~$i \in I$}\\ 
  B_{i} & \mbox{otherwise}.
\end{cases} 
\end{equation}
Note that $h_I$ is of degree~$\# I$ in $z_1,\ldots,z_n$.
Hence the homogeneous part~$g_k$ satisfies} 
\begin{equation}\label{gkSum}
  g_{k}(z_{0},\dotsc,z_{\delta}) = \sum_{m = 1}^k  \binom{\delta - m}{k-m}z_{0}^{k-m}
  \sum_{\# I = m} h_I(z_1,\dotsc,z_n).
\end{equation}
The contribution for $m=0$ vanishes by assumption:
\begin{equation}
  \binom{\delta }{k}z_{0}^{k} h_\varnothing = \binom{\delta }{k}z_{0}^{k} \tr (B_1\dotsb B_k) = 0 .\label{eq:77}
\end{equation}
All the terms of the outer sum in \eqref{gkSum} over~$m$ have disjoint monomial support,
so they are orthogonal for the Weyl {inner product}; see \S\ref{se:weyl}. 
Moreover for any homogeneous polynomial~$p(z_1,\dotsc,z_n)$ of degree~$m\leq k$,
the definition of the Weyl norm easily implies $\binom{k}{m}\|z_0^{k-m} p\|^2_W = \| p \|_W^2$.
It follows that
\begin{equation}
  \left\| g_k \right\|_W^2 = \sum_{m=1}^k \binom{\delta - m}{k-m}^2
  \binom{k}{m}^{-1} \left\|\sum_{\# I = m} h_I \right\|^2_W.
\end{equation}
For two different subsets~$I,I'\subseteq \{1,\dotsc,\delta\}$,
there is at least one index~$i$ such that~$C_i$ occurs in~$h_I$ and not
in~$h_{I'}$, so that {the Weyl inner product}~$\langle h_I, h_{I'} \rangle_{{W}}$
depends linearly on~$C_i$ and then, by symmetry,
$\int \ud' C \left\langle h_I, h_{I'}\right\rangle_W  = 0$.
It follows that
\begin{equation}\label{eq:15}
  \int \ud'C \ \left\| g_k \right\|_W^2
  = \sum_{m=1}^k \binom{\delta - m}{k-m}^2 \binom{k}{m}^{-1}
  \sum_{\# I = m}  \int \ud'C \left\| h_I \right\|^2_W.
\end{equation}

For computing~$\int \ud' C \left\| h_I \right\|^2_W$, with~$\# I = m > 0$,
we proceed as follows. From Lemma~\ref{le:49}
($h_I$ is a homogeneous polynomial in $n$ variables of degree~$m$), we obtain that
\begin{equation}\label{eq:13}
  \left\| h_I \right\|_W^2
  = \binom{m+n-1}{m} \frac{1}{\vol \mathbb{S}(\mathbb{C}^{n})}
  \int_{\mathbb{S}(\mathbb{C}^n)} \ud z \left| h_I(z) \right|^2.
\end{equation}
Then, given that the tuple~$(C_1,\dotsc,C_\delta)$ is standard Gaussian
in $M_\bfr(n)$,
the matrices~$C_1(z), \dotsc, C_\delta(z)$ are
independent standard Gaussian random matrices,
for any~$z \in \mathbb{S}(\mathbb{C}^r)$.
Let~$I\subseteq \{1,\dotsc,\delta\}$ be such that~$1 \in I$  
(without loss of generality,  
because the indices are defined up to cyclic permutation).
Then we have
$h_I(z_1,\dotsc,z_n) = \tr \left( C_1(z) U^{I}_{2} \dotsb U^{I}_{\delta} \right)$.
Integrating over~$C_1$,
Lemma~\ref{lem:int-tr}\ref{item:int-tr:1} shows 
for a fixed~$z\in \mathbb{S}(\mathbb{C}^{n+1})$ that 
\begin{align}
  \label{eq:7}
  \int \ud' C_1 \ \left| h_I(z) \right|^2 = \|U^I_{2}\dotsb U^I_{\delta}\|^2_\Frob.
\end{align}
Integrating further with
respect to~$C_i$ with~$i \not\in I$ is trivial since $\|U_2^I\dotsm U_\delta^I\|^2_\Frob$
does not depend on these~$C_i$. 
To integrate with respect to~$C_i$ with~$i \in I$,
we use Lemma~\ref{lem:int-tr}\ref{item:int-tr:2} to obtain
\begin{equation}
  \int \ud' C_1 \ud'C_i \ \left| h_I(z) \right|^2
  = \|U_2^I \dotsm U_{i-1}^I \|^2_\Frob \;\|U^I_{i+1} \dotsm U_\delta^I\|^2_\Frob.
\end{equation}
After integrating with respect to {the remaining~$C_i$} in the same way, we obtain
\begin{equation}\label{eq:14}
  \int \ud' C \ \left| h_I(z) \right|^2 = P_I(B),
\end{equation}
where~$P_I(B)$ does not depend on~$z$ and is defined as follows.
Let~$I=\{i_1,\dotsc,i_m\}$, with~$1=i_1<\dotsb<i_m$.
Then 
\begin{equation}
  P_I(B) \eqdef \|B_{2} \dotsb B_{i_2-1}\|_\Frob^2 \|B_{i_2+1}\dotsb B_{i_3-1}\|^2_\Frob
  \dotsb \|B_{i_m+1} \dotsb B_{\delta}\|^2_\Frob.
\end{equation}
More generally, if~$i_1 \neq 1$, $P_I(B)$ is defined as above with
the first and last factors replaced, respectively, by
\begin{equation}
  \|B_{i_1+1} \dotsb B_{i_2-1}\|^2_\Frob  \text{ and } \|B_{i_m + 1}
  \dotsb B_{\delta} B_1 \dotsb B_{i_1-1}\|^2_\Frob,
\end{equation}
and~\eqref{eq:14} still holds.
{Averaging~\eqref{eq:14} with respect to $z\in {\mathbb S}(\bC^{n})$,
we obtain with~\eqref{eq:13}}
\begin{equation}
  \int \ud'C \ \|h_I\|_W^2 = \binom{m + n-1}{m} P_I(B).
\end{equation}
Combining further with~\eqref{eq:16} and~\eqref{eq:15}, we obtain
\begin{equation}\label{eq:22}
  \int \ud' C\ \left\|{\tfrac{1}{k!}}\ud_\zeta^k f_A \right\|_\Frob^2
  = \sum_{m=1}^k \binom{\delta - m}{k-m}^2 \binom{k}{m}^{-1} \binom{m+n-1}{m}
  \sum_{\# I = m}  P_I(B).
\end{equation}
Combining with~\eqref{eq:17}, this leads to
\begin{multline}\label{eq:6}
  \smashint[\ud' A]{M_\zeta}\  \left\|\partial_{A} f(\zeta) \right\|^{-2}
  \left\|{\tfrac{1}{k!}}\ud_{\zeta}^{k} f_{A}\right\|^{2} = \\
  \sum_{m=1}^k \binom{\delta - m}{k-m}^2 \binom{k}{m}^{-1} \binom{m+n-1}{m}
  \sum_{\#I = m} \smashint[\ud'B]{W} \frac{P_I(B)}{\| \hat B_1 \|^2_\Frob + \dotsb
    + \| \hat B_\delta \|^2_\Frob}.
\end{multline}
Recall that~$\hat B_i = B_{i+1}\dotsb B_\delta B_1 \dotsb B_{i-1}$.

\subsubsection{Computation of the integral over~$W$}

We now consider the integral
\begin{equation}
  \smashint[\ud'B]{W} \frac{P_I(B)}{\| \hat B_1 \|^2_\Frob + \dotsb + \| \hat B_\delta \|^2_\Frob},
\end{equation}
which appears in the right-hand side of~\eqref{eq:6}.
The goal is the bound~\eqref{eq:19}.
To simplify notation, we assume~$1\in I$ but this does not change anything,
up to cyclic permutation of the indices.
We apply the coarea formula to the projection
$q\colon W\to F,\,  B \mapsto (B_2,\dotsc,B_\delta)$,  
where~$F \eqdef \mathbb{C}^{r_1\times r_2}\times \dotsb \times \mathbb{C}^{r_{\delta-1}\times r_\delta}$.
{Since the complex hypersurface~$W$ is defined by the condition~$\tr(B_1\dotsb B_\delta)=0$,}
we have 
\begin{equation}
  T_{B} W = \left\{ (\dot B_1,\dotsc,\dot B_\delta) \st \sum_{i=1}^\delta
  \tr(\dot B_i \hat B_i) =0 \right\} \subseteq \bC^{r_\delta\times r_1} 
\times F ;
\end{equation}
this is the same computation as for~\eqref{eq:58}.
In particular, the normal space of $W$ 
is spanned by~$(\hat B_1^*,\dotsc,\hat B_\delta^*)$,
where~${}^*$ denotes the Hermitian transpose.
It follows from Lemma~\ref{lem:quotien-of-NJ} (used as in Corollary~\ref{cor:NJ-quot}) that the normal Jacobian~{$\mop{NJ}_B(q)$}
of~{$q$} at some~$B\in W$ is given by
\begin{equation}
  \mop{NJ}_B(q) = \frac{\|\hat B_1\|^2}{\|\hat B_1\|^2_\Frob + \dotsb + 
\|\hat B_\delta\|^2_\Frob} . 
\end{equation}
The coarea formula then gives
\begin{equation}
  \smashint[\ud'B]{W} \frac{P_I(B)}{\| \hat B_1 \|^2_\Frob + \dotsb + \| \hat B_\delta \|^2_\Frob}
  = \smashint[\ud'B_2 \dotsb \ud'B_\delta]{F}
  \smashintlong[\ud 'B_1]{\tr(B_1\dotsb B_\delta)=0}
  \frac{P_I(B)}{\|\hat B_1 \|^2_\Frob}.
\end{equation}

Note that the inner integrand does not depend on~$B_1$. 
{Moreover, for fixed $B_2,\ldots,B_\delta$, the condition~$\tr(B_1\dotsb B_\delta)=0$
restricts~$B_1$ to a hyperplane in~$\mathbb{C}^{r_0\times r_1}$. 
Due to the unitary invariance of the standard Gaussian measure, the position of the
hyperplane does not matter} and we obtain
\begin{equation}
  \int_{\tr(B_1 \dotsb B_\delta)=0} \ud' B_1 =  \int_{{\bC^{r_0r_1-1}}} \ud' B_1 = \frac{1}{\pi}.
\end{equation}
It follows that
\begin{align}\label{eq:21}
  \MoveEqLeft \smashint[\ud'B]{W} \frac{P_I(B)}{\| \hat B_1 \|^2_\Frob + \dotsb
    + \| \hat B_\delta \|^2_\Frob}
  = \frac1\pi \smashint[\ud'B_2 \dotsb \ud'B_\delta]{F}\
    \frac{P_I(B)}{\|\hat B_1 \|^2_\Frob}\\
    &= \frac1\pi \smashint[\ud'B_2 \dotsb \ud'B_\delta]{F}
    \frac{\prod_{k=1}^{m-1} \|B_{i_k+1} \dotsb B_{i_{k+1}-1}\|^2_\Frob
      \cdot  \|B_{i_m+1} \dotsb B_{\delta}\|^2_\Frob}
      {\|B_2 \dotsb B_{i_2} \dotsb B_{i_3} \dotsb \dotsb B_{i_m} \dotsb B_\delta\|^2_\Frob},
\end{align}
where~$I= \left\{ i_1,\dotsc,i_m \right\}$ with $i_1=1$.
If~$m= 1$, that is~$I= \left\{ 1 \right\}$, then the integrand
simplifies to~$1$. 

{Recall the anomaly $\theta(A)$ of a matrix defined in~\eqref{eq:def-anomaly}.}
When~$m > 1$, we take expectations over $B_{i_2},\ldots,B_{i_m}$ and repeatedly apply
Lemma~\ref{lem:anomaly}, to obtain\footnotemark
\footnotetext{
Let us exemplify the computations \eqref{eq:153}--\eqref{eq:18} on a particular case: $\delta=6$
and~$I= \left\{ 1,4,5 \right\}$.
In this case~$P_I(B) = \|B_2 B_3\|_\Frob^2 \|\mathbf{1}\|_\Frob^2 \|B_6\|_\Frob^2$,
where~$\mathbf 1$ is the identity matrix of size~$r_4\times r_ 4$. Then,
by~\eqref{eq:21} and two applications of Lemma~\ref{lem:anomaly} (first for
integrating w.r.t~$B_4$ then~$B_5$),
\begin{align*}
  \smashint[\ud'B_2 \dotsb \ud'B_6]{}\ \frac{P_I(B)}{\|\hat B_1 \|^2}
  &= \smashint[\ud'B_2 \dotsb \ud 'B_6]{} \frac{\|B_2 B_3\|_\Frob^2\; \|B_5B_6\|_\Frob^2}
  {\|B_2 B_3 B_4B_5B_6\|^2_\Frob} \frac{\|\mathbf 1\|_\Frob^2\, \|B_6\|^2}{\|B_5B_6\|^2_\Frob}\\
  &\leq \smashint[\ud'B_2\, \ud'B_3\, \ud'B_5\, \ud 'B_6]{E} \ \theta(B_2B_3)
  \frac{\|\mathbf 1\|_\Frob^2\, \|B_6\|^2}{\|B_5B_6\|^2_\Frob}\\
  &\leq \left( \smashint[\ud'B_2\, \ud'B_3]{} \ \theta(B_2B_3) \right) \theta(\mathbf 1)\\
  &= \left(1 + \frac{1}{r_1-1}+\frac{1}{r_2-1} + \frac{1}{r_3-1} \right)
  \left( 1+\frac{1}{r_4-1} \right),
\end{align*}
the last by Lemma~\ref{lem:anomaly-product}.
}
\begin{align}\label{eq:153}
  \smashint[\ud'B_2 \dotsb \ud' B_\delta]{F}\ \frac{P_I(B)}{\| \hat B_1 \|^2_\Frob}
  &\leq \prod_{k=1}^{m-1} \ \int \ud'B_{i_k+1}\dotsb \ud'B_{i_{k+1}-1}
  \theta(B_{i_k+1}\dotsb B_{i_{k+1}-1}).
\end{align}
Every block $B_{i_k+1}\dotsb B_{i_{k+1}-1}$ appears except the last
block~$B_{i_m+1}\dotsb B_\delta$.  If one of the parameters~$r_i$ is~1, 
then, by cyclic permutation of the indices, we 
may assume that it appears in the last block 
(indeed, by the hypothesis 
$r_1,\ldots,r_{\delta-1} \ge 2$, 
there is at most one $i$ with $r_i=1$).

So we can apply
Lemma~\ref{lem:anomaly-product} and obtain
\begin{align}
  \smashint[\ud'B_2 \dotsb \ud' B_\delta]{F}\ \frac{P_I(B)}{\| \hat B_1 \|^2_\Frob}
  \label{eq:18} &\leq\prod_{j=1}^{m-1} \left( 1 + \sum_{j=i_k}^{i_{k+1}-1}
  \frac{1}{r_j-1} \right) \\
  &\leq \left( 1 + \frac{1}{m-1} \sum_{j=i_1}^{i_m-1} \frac{1}{r_j-1}  \right)^{m-1}
  \label{eq:44},
\end{align}
using the inequality of arithmetic
and geometric means. 
Since $r_j > 1$ for $j\leq i_m-1$, we further obtain
\begin{align}\label{eq:19}
  \smashint[\ud'B_2 \dotsb \ud' B_\delta]{F}\ \frac{P_I(B)}{\| \hat B_1 \|^2_\Frob}
  \leq \left( 1+\frac{\delta-1}{m-1} \right)^{m-1}.
\end{align}

\subsubsection{Conclusion}
Combining~\eqref{eq:20}, \eqref{eq:6},  \eqref{eq:21}, and~\eqref{eq:19}, 
we obtain
(note the cancellation of $\pi$),
\begin{multline}\label{eq:23}
  \mathbb{E} \left[ \left\| \ud_{\zeta} f_{A} \right\|^{-2}
    \left\|{\tfrac{1}{k!}}\ud_{\zeta}^{k} f_{A}\right\|_\Frob^{2}\right]\\
  \leq \frac{1}{\delta n} \sum_{m=1}^k \binom{\delta - m}{k-m}^2
  \binom{k}{m}^{-1} \binom{m+n-1}{m}
  \binom{\delta}{m} \left( 1 + \frac{\delta-1}{m-1} \right)^{m-1}.
\end{multline}
By reordering the factorials, we have
\begin{equation}\label{eq:24}
  \binom{\delta-m}{k-m}^2 \binom{k}{m}^{-1} \binom{\delta}{m}
  = \binom{\delta-m}{k-m} \binom{\delta}{k}.
\end{equation}
As in the proof of Lemma~I.37, we observe the identity
\begin{equation}\label{eq:25}
  \sum_{m=0}^k \binom{\delta-m}{k-m} \binom{m+n-1}{m}
  = \binom{\delta+n}{k}.
\end{equation}
Moreover, since~$m \leq k$,
\begin{equation}\label{eq:26}
  \left( 1 + \frac{\delta-1}{m-1} \right)^{m-1}
  \leq \left( 1 + \frac{\delta-1}{k-1} \right)^{k-1}
\end{equation}
(including~$m=1$ where the left-hand side is~1).
Equations~\eqref{eq:23}, \eqref{eq:24}, \eqref{eq:25} and~\eqref{eq:26} give
\begin{equation}
  \mathbb{E} \left[ \left\| \ud_{\zeta} f_{A} \right\|^{-2}
    \left\|{\tfrac{1}{k!}}\ud_{\zeta}^{k} f_{A}\right\|_\Frob^{2}\right]
  \leq \frac{1}{n\delta} \binom{\delta}{k} \binom{\delta+n}{k}
  \left( 1 + \frac{\delta-1}{k-1} \right)^{k-1}.
\end{equation}
This gives the first inequality of Proposition~\ref{prop:main-prop-stochastic}.
For the second we argue as in the proof of Lemma~I.37: 
{the maximum value of
$\left[\frac{1}{n\delta} \binom{\delta}{k} \binom{\delta+n}{k}\right]^{\frac{1}{k-1}}$
with $k\ge 2$ is reached at $k=2$. Hence, for any $k\ge 2$,}
\begin{equation}
 \left[\frac{1}{n\delta} \binom{\delta}{k} \binom{\delta+n}{k}\right]^{\frac{1}{k-1}}
 \le 
 \left[\frac{1}{n\delta} \binom{\delta}{2} \binom{\delta+n}{2}\right]^{\frac{1}{k-1}}
 \le  
 \left[\frac{1}{4} \delta^2 (\delta+n) \right]^{\frac{1}{k-1}}.
\end{equation}
This concludes the proof of Proposition~\ref{prop:main-prop-stochastic}.

\printbibliography

\end{document}